\documentclass[a4paper,reqno,11pt]{amsart}

\usepackage{isolatin1}
\usepackage{amssymb}
\usepackage[all,cmtip]{xy}
\usepackage{tikz}

\allowdisplaybreaks

\theoremstyle{plain}
\newcounter{theorem}[section]
\numberwithin{equation}{section}

\newtheorem{theorem}[theorem]{Theorem}

\newtheorem{corollary}[theorem]{Corollary}

\newtheorem{lemma}[theorem]{Lemma}
\newtheorem{proposition}[theorem]{Proposition}
\theoremstyle{definition}
\newtheorem{remark}[theorem]{Remark}
\newtheorem{algo}[theorem]{Algorithm}    
\newtheorem*{hyp}{Hypotheses}

\newcommand{\smod}{\!\!\!\mod}

\usepackage{hyperref}

\newcommand{\mbinom}[2]{\genfrac{[}{]}{0pt}{}{#1}{#2}}

\newcommand{\sectm}[2]{\section{\texorpdfstring{#1}{#2}}}
\newcommand{\subsectb}[2]{\subsection{\texorpdfstring{\boldmath#1}{#2}}}

\newcommand{\msp}{\mspace{1mu}}
\newcommand{\mspp}{\mspace{2mu}}
\newcommand{\tmod}[1]{\allowbreak\ ({\textup{mod}}\,\,#1)}
\renewcommand{\div}{\mathbin{\mid}}
\newcommand{\notdiv}{\mathbin{\nmid}}

\newcommand{\uppar}[1]{\textup{(}#1\textup{)}}
\newcommand{\ldss}[2]{#1,\,\ldots,\,#2}
\newcommand{\ldsss}[3]{#1,\,#2,\,\dots,\,#3}
\newcommand{\range}[3]{#1=\ldss{#2}{#3}}
\newcommand{\emphq}[1]{``#1''}
\newcommand{\tO}{O\tilde{\ }}
\newcommand{\tF}{\widetilde{F}}
\newcommand{\tf}{\tilde{f}}
\newcommand{\ts}{\tilde{s}}
\newcommand{\ta}{\tilde{a}}
\newcommand{\tmu}{\tilde{\mu}}
\newcommand{\tgamma}{\tilde{\gamma}}
\newcommand{\tsigma}{\tilde{\sigma}}
\newcommand{\p}{p\msp}
\newcommand{\toprec}[1]{to the precision $#1$}
\newcommand{\vdepth}[1]{\vrule width 0pt height 0pt depth #1}
\newcommand{\Cl}{\operatorname{Cl}}
\newcommand{\Gal}{\operatorname{Gal}}
\newcommand{\Ker}{\operatorname{Ker}}
\newcommand{\res}{\operatorname{res}}
\newcommand{\Coeff}{\operatorname{Coeff}}
\newcommand{\RC}{RC}
\newcommand{\PC}{PC}

\begin{document}

\title[Computing $\p$-adic $L$-functions of totally real number fields]
      {\boldmath Computing $p\mspp$-adic $L$-functions of totally real number fields}

\author{Xavier-François Roblot}

\address{Université de Lyon, Université Lyon 1, CNRS UMR5208, Institut Camille Jordan, 43 blvd du 11 novembre 1918, 69622 Villeurbanne Cedex, France}
\email{roblot@math.univ-lyon1.fr}
\address{Tokyo Institute of Technology, Department of Mathematics, 2-12-1 Ookayama, Meguro-ku, Tokyo 
152-8550, Japan}
\email{roblot@math.titech.ac.jp}

\thanks{Supported in part by the ANR \texttt{AlgoL} (ANR-07-BLAN-0248) and the JSPS Global COE \emph{CompView}.}

\date{\today \ (draft v.4)}

\maketitle

\begin{abstract}
We prove explicit formulas for the $p$-adic $L$-functions of totally real number fields and show how these formulas can be used to compute values and representations of $p$-adic $L$-functions. 
\end{abstract}

\section{Introduction}

The aim of this article is to present a general method for computing values and representations of $\p$-adic $L$-functions of totally real number fields.\footnote{One can prove that $\p$-adic $L$-functions of non-totally real number fields are identically zero.} These functions are the $\p$-adic analogues of the ``classical'' complex $L$-functions and are related to those by the fact that they agree, once the Euler factors at $p$ have been removed from the complex $L$-functions, at negative integers in some suitable congruence classes. The existence of $\p$-adic $L$-functions was first established in 1964 by Kubota-Leopoldt \cite{kubota-leopoldt} over $\mathbb{Q}$ and consequently over abelian extensions of $\mathbb{Q}$. It was proved in full generality, 15 years later, by Deligne-Ribet \cite{dr:zeta} and, independently, by Barsky \cite{barsky:zeta} and Cassou-Noguès \cite{cn:zeta}. The interested reader can find a summary of the history of their discovery in \cite{cn:zeta}. 

There have already appeared many works on the computation of $\p$-adic $L$-functions, starting with Iwasawa-Sims \cite{iwasawa-sims} in 1965 (although they are not explicitly mentioned in the paper) to the more recent computational study of their zeroes by Ernvall-Mets{\"a}nkyl{\"a} \cite{em1, em2} in the mid-1990's and the current work of Ellenberg-Jain-Venkatesh \cite{ejv} that provides a conjectural model for the behavior of the $\lambda$-invariant of $\p$-adic $L$-functions in terms of properties of $\p$-adic random matrices. However, most of these articles deal only with $L$-functions over $\mathbb{Q}$ or that can be written as a product of such $L$-functions. One remarkable exception is the work of Cartier-Roy \cite{cartier-roy}  in 1972 where computations were carried on to support the existence (at the time not yet proven) of $\p$-adic $L$-functions over some non-abelian fields of degree $3$, $4$ and~$5$. 

The method for computing $\p$-adic $L$-functions given in the present paper is derived from the construction found in \cite{cn:zeta, katz:measures, lang:cyclo}. It generalizes a previous work with Solomon \cite{RS1}. The idea is the following. First, using Shitani's cone decomposition (see Subsection~\ref{subsec:conedec}), we express $L$-functions in terms of cone zeta functions (see Subsection~\ref{subsec:zetcone}, Proposition~\ref{prop:lfunc} and Equation~\ref{eq:zetacones}). Then, for a given cone zeta function, its values at negative integers are encoded into a power series (see Subsections~\ref{subsec:zetcone} and \ref{subsec:omega}). Using the method of Section~\ref{sec:padic}, this power series is then interpreted as a $\p$-adic measure. The $\p$-adic cone zeta function is obtained by integrating suitable $\p$-adic continuous functions against this measure (see Theorem~\ref{th:procinterpol}) once it is proved that it satisfies the required properties (see Subsection~\ref{subsec:propF}). The main tool for the computation is an explicit formula for the power series associated with the cone zeta function, up to a given precision; see Theorem~\ref{th:main}. From this formula we give an explicit expression for the values of the cone zeta function at some $\p$-adic integer (Theorem~\ref{th:componeval}) and an explicit expression for the corresponding Iwasawa power series (Theorem~\ref{th:compiwas}). Note that these also are valid only up to a given precision. 

One shortcoming of our method is that it is not very efficient compared to the complex case (see Subsection~\ref{subsec:compL} for some complexity estimates). For example, in this simplest case of $\p$-adic $L$-functions over $\mathbb{Q}$, for a Dirichlet character of conductor $f$, the complexity in $f$ of the method presented here is $O(f^{1+\epsilon})$, whereas there exist methods to compute complex Dirichlet $L$-functions in $O(f^{1/2+\epsilon})$. Even in this simple case it remains an open problem whether methods as efficient exist in the $\p$-adic case.

\smallskip

\noindent\textbf{Note.} The construction presented in this paper was developed over several years and during that time was used in two previous works; see \cite{bbjr, rw}. It is worth noting that the method has evolved and therefore the brief description of it in these earlier articles does not necessarily match exactly the one that is finally presented here.  

\smallskip

\noindent\textbf{Acknowledgments.} I would like to express my gratitude to D.~Solomon for his invaluable help with this project and for introducing me to the fascinating world of $\p$-adic $L$-functions. I am also grateful to D.~Barsky for his help at the start of the project and to R.~de~Jeu and A.~Weiss for providing with the motivation to complete this project. Finally, I thank heartily D.~Ford for reading carefully the earlier version of this article and for his many corrections and helpful comments.

\sectm{$\p$-adic interpolation}{p-adic interpolation}\label{sec:padic}

Let $p$ be a prime number. Denote by $\mathbb{Q}_p$ the field of rational $\p$-adic numbers. The subring of $\p$-adic integers is denoted by $\mathbb{Z}_p$, and $\mathbb{C}_p$ is the completion of the algebraic closure of $\mathbb{Q}_p$. Let $|\cdot|_p$ denote the $\p$-adic absolute value of $\mathbb{C}_p$ normalized so that $|p|_p = p^{-1}$ and $v_p(.)$ the corresponding valuation; thus $v_p(p) = 1$. For $f \geq 1$, an integer, let $W_f$ denote the subgroup of $f$-th roots of unity in $\mathbb{C}_p$. The torsion part $T_p$ of the group $\mathbb{Z}_p^\times$ of units in $\mathbb{Z}_p$ is equal to $W_{\varphi(q)}$ where $q := 4$ if $p = 2$, $q := p$ if $p$ is odd, and $\varphi$ is Euler totient function. We have $\mathbb{Z}_p^\times = T_p \times (1 + q\mathbb{Z}_p)$, and the projections $\omega : \mathbb{Z}_p^\times \to T_p$ and $\langle \cdot \rangle : \mathbb{Z}_p^\times \to 1 + q\mathbb{Z}_p^\times$ are such that $x = \omega(x) \langle x \rangle$ for all $x \in \mathbb{Z}_p^\times$. In particular, we have $x \equiv \omega(x) \tmod{q}$ for all $x \in \mathbb{Z}_p^\times$. 

\subsectb{Continuous $\p$-adic functions}{Continuous p-adic functions}

For $n \in \mathbb{N} := \{0, 1, 2, \dots\}$, the \emph{binomial polynomial} is defined by 
\begin{equation}\label{eq:binom}
\binom{x}{n} := 
\begin{cases}
1 & \text{if } n = 0, \\
\dfrac{x(x - 1) \cdots (x - (n - 1))}{n!} & \text{otherwise}.
\end{cases}
\end{equation}
The binomial polynomial takes integral values on $\mathbb{Z}$, hence, by continuity, it takes $\p$-adic integral values on $\mathbb{Z}_p$. Let $f$ be a function on $\mathbb{Z}_p$ with values in $\mathbb{C}_p$. One can easily construct by induction a sequence $(f_n)_{n \geq 0}$ of elements of $\mathbb{C}_p$ (see Subsection~\ref{subsec:compcontfunc}) such that
\begin{equation}\label{eq:mahlerexp}
f(x) = \sum_{n \geq 0} f_n \binom{x}{n} \quad \text{for all } x \in \mathbb{N}.
\end{equation}
(Note that this is in fact a finite sum.) The coefficients $f_n$'s are uniquely defined and are called the \emph{Mahler coefficients} of $f$.  We have the following fundamental result (see \cite[\S4.2.4]{robert}). 
\begin{theorem}[Mahler expansion]\label{th:mahler}
Let $f$ be a function on $\mathbb{Z}_p$ with values in $\mathbb{C}_p$. Then $f$ is continuous on $\mathbb{Z}_p$ if and only if 
\begin{equation*}
\lim\limits_{n \to \infty} |f_n|_p = 0.
\end{equation*}
If $f$ is continuous, then the sequence of continuous functions 
\begin{equation}\label{eq:seqfunc}
x \mapsto \sum_{n = 0}^N f_n \binom{x}{n}
\end{equation}
converges uniformly to $f$. Reciprocally, let $(f_n)_{n \geq 0}$ be a sequence of elements in $\mathbb{C}_p$ converging to zero. Then the sequence of functions in \eqref{eq:seqfunc} above 
converges to a continuous function.
\end{theorem}
Denote by $\mathcal{C}(\mathbb{Z}_p, \mathbb{C}_p)$ the set of continuous functions on $\mathbb{Z}_p$ with values in $\mathbb{C}_p$. For $f \in \mathcal{C}(\mathbb{Z}_p, \mathbb{C}_p)$, we define the \emph{norm} of $f$ by
\begin{equation*}
\|f\|_p := \max_{x \in \mathbb{Z}_p} |f(x)|_p.
\end{equation*}
The norm of $f$ is a finite quantity since $\mathbb{Z}_p$ is compact and in fact, if $(f_n)_{n \geq 0}$ are the Mahler coefficients of $f$, we have
\begin{equation}\label{eq:normfromcoeff}
\|f\|_p = \max_{n \geq 0}|f_n|_p.
\end{equation}

\subsection{A family of continuous functions}

We define a family of continuous functions that will be useful later on. For $s \in \mathbb{Z}_p$, we would like to define $x \mapsto x^s$, where $x$ is an $\p$-adic number, in such a way to extend the definition of $x \mapsto x^k$ when $s = k \in \mathbb{Z}$. In general, it is not possible. However, when $x \in 1 + q\mathbb{Z}_p$, one can set
\begin{equation*}
x^s := \sum_{n \geq 0} (x - 1)^n \binom{s}{n}.
\end{equation*}
The series converges since $|x-1|_p < 1$, and by Theorem~\ref{th:mahler}, the function $s \mapsto x^s$ is continuous.\footnote{Actually, for $p = 2$ we only need $x \in 1 + 2\mathbb{Z}_2$. But in order to have an analytic function it is necessary to assume $x \in 1 + 4\mathbb{Z}_2$; see Subsection~\ref{subsec:analyticity}.}
Furthermore, when $s = k \in \mathbb{N}$ we recover the usual definition of $x^k$ by the binomial theorem, and it is easy to see that this function has the expected properties. With that in mind, for $s\in \mathbb{Z}_p$ we define the function $\phi_s$ in $\mathcal{C}(\mathbb{Z}_p, \mathbb{C}_p)$ by
\begin{equation*}
\phi_s(x) := 
\begin{cases}
0 & \text{ if } x \in p\mathbb{Z}_p, \\
\langle x \rangle^s & \text{ if } x \in \mathbb{Z}_p^\times.
\end{cases}
\end{equation*}
It it easy to see that $\phi_s$ is a continuous function and that its restriction to $\mathbb{Z}_p^\times$ is a group homomorphism. We state two results concerning the properties of $\phi_s$. The first one follows directly from construction.

\begin{lemma}\label{lem:expfact}
Let $k$ be a integer. Then, for all $x \in 1+q\mathbb{Z}_p$, we have
\begin{equation*}
\phi_k(x) = x^k.  \tag*{\hspace{-1em}\qed}
\end{equation*}
\end{lemma}

\begin{lemma}\label{lem:phicont}
The map $s \mapsto \phi_s$ from $\mathbb{Z}_p$ to $\mathcal{C}(\mathbb{Z}_p, \mathbb{C}_p)$ is continuous.
\end{lemma}
\begin{proof}
Let $s, s'$ in $\mathbb{Z}_p$. If $x \in p\mathbb{Z}_p$, then $\phi_s(x) = \phi_{s'}(x) = 0$. If $x \in \mathbb{Z}_p^\times$, then 
\begin{align*}
|\phi_s(x) - \phi_{s'}(x)|_p
& = \big|\langle x \rangle^s \big(1 - \langle x \rangle^{s'-s}\big)\big|_p
   = \biggl|1 - \sum_{n \geq 0} \binom{s' - s}{n} (x - 1)^n\biggr|_p \\ 
& \leq |s' - s|_p \left|\frac{(x-1)^n }{n!}\right|_p \leq |s' - s|_p.
\qedhere
\end{align*}
\end{proof}

\subsectb{Integration of $\p$-adic continuous functions}{Integration of p-adic continuous functions}\label{subsec:integration}

A \emph{measure} $\mu$ on $\mathbb{Z}_p$ is a bounded linear functional on the $\mathbb{C}_p$-vector space $\mathcal{C}(\mathbb{Z}_p, \mathbb{C}_p)$. That is, there exists a constant $B > 0$ satisfying
\begin{equation}\label{eq:mubnd}
|\mu(f)|_p \leq B \, \|f\|_p \quad \text{for all } f \in \mathcal{C}(\mathbb{Z}_p, \mathbb{C}_p).
\end{equation}
The smallest possible $B$ is called the \emph{norm} of the measure $\mu$ and is denoted $\|\mu\|_p$. With this norm, the set $\mathcal{M}(\mathbb{Z}_p, \mathbb{C}_p)$ of measures on $\mathcal{C}(\mathbb{Z}_p, \mathbb{C}_p)$ becomes a $\mathbb{C}_p$-Banach space. From now on, we will write
\begin{equation*}
\int f(x) \, d\mu(x) := \mu(f).
\end{equation*}
Usually we will drop the $x$ to simplify the notation when the context is clear. 

\begin{lemma}\label{lem:intcont}
The function $\mu$ is a continuous map from $\mathcal{C}(\mathbb{Z}_p, \mathbb{C}_p)$ to $\mathbb{Z}_p$.
\end{lemma}
\begin{proof}
This is clear by \eqref{eq:mubnd}.
\end{proof}

\begin{lemma}\label{lem:dointeg}
Let $f \in \mathcal{C}(\mathbb{Z}_p, \mathbb{C}_p)$ with Mahler coefficients $(f_n)_{n \geq 0}$. Then we have
\begin{equation}
\int f \, d\mu = \sum_{n \geq 0} f_n \int \binom{x}{n} \, d\mu.
\end{equation}
\end{lemma}
\begin{proof} 
This is clear since $f = \lim\limits_{N \to \infty} \sum\limits_{n = 0}^N f_n \dbinom{x}{n}$ and $\mu$ is continuous by the previous lemma.
\end{proof}
For $\mu \in \mathcal{M}(\mathbb{Z}_p, \mathbb{C}_p)$ and $g \in \mathcal{C}(\mathbb{Z}_p, \mathbb{C}_p)$, the measure $g\mu$ is defined, for  any $f \in \mathcal{C}(\mathbb{Z}_p, \mathbb{C}_p)$, by 
\begin{equation*}
\int f(x) \, dg\mu(x) := \int f(x)g(x) \, d\mu(x).
\end{equation*}
When $g = \chi_A$, the characteristic function  of an open and closed subset $A$ of $\mathbb{Z}_p$, we will use the notation
\begin{equation*}
\int_A f \, d\mu := \int f \, d\chi_A\mu.
\end{equation*}
A measure $\mu$ is said to have \emph{support in} $A$ if $\mu = \chi_A\mu$. In other words, for all $f \in  \mathcal{C}(\mathbb{Z}_p, \mathbb{C}_p)$ we have
\begin{equation*}
\int_A f \, d\mu := \int f \, d\mu.
\end{equation*}

\subsection{Measures and power series}

Let $\mathbb{C}_p[[T]]^{bd}$ be the $\mathbb{C}_p$-algebra of power series whose coefficients are in $\mathbb{C}_p$ and are bounded in absolute value. Let $\mu$ be a measure in $\mathcal{M}(\mathbb{Z}_p, \mathbb{C}_p)$. One associates with $\mu$ a power series $F_\mu \in \mathbb{C}_p[[T]]^{bd}$ defined by
\begin{equation*}
F_\mu(T) := \sum_{n \geq 0} \int \binom{x}{n} \, d\mu(x) \, T^n.
\end{equation*}
Reciprocally, given a power series $F \in \mathbb{C}_p[[T]]^{bd}$ with coefficients $F_n$ ($n \geq 0$), one associates with $F$ a measure $\mu_F$ defined by
\begin{equation}\label{eq:Ftomu}
\sum_{n \geq 0} \int \binom{x}{n} \, d\mu_F(x) := F_n.
\end{equation}
Indeed, by Lemma~\ref{lem:dointeg}, these equations uniquely determine the measure $\mu_F$. These maps define isometric isomorphisms of $\mathbb{C}_p$-Banach space between $\mathcal{M}(\mathbb{Z}_p, \mathbb{C}_p)$ and $\mathbb{C}_p[[T]]^{bd}$ where the norm on $\mathbb{C}_p[[T]]^{bd}$ is defined to be the maximum of the absolute values of the coefficients; see \cite[Chap.~4]{lang:cyclo}.
\begin{remark}\label{rem:uniqfmu}
Another characterization of the correspondence between measures and bounded power series is that the power series $F_\mu$ is the unique power series in $\mathbb{C}_p[[T]]^\text{bd}$ such that 
\begin{equation*}	
\int (1 + t)^x \, d\mu(x) = F_\mu(t) \quad\text{for all } t \in \mathbb{C}_p \text{ such that } |t|_p < 1.
\end{equation*}	
\end{remark}

The measures corresponding to powers of $1+T$ form an important class. The result below follows directly from \eqref{eq:seqfunc}, Lemma~\ref{lem:dointeg} and \eqref{eq:Ftomu} (or simply the remark above).
\begin{lemma}\label{lem:dirac}
Let $a \in \mathbb{Z}_p$. Then the measure associated with the power series 
\begin{equation}
(1+T)^a := \sum_{n \geq 0} \binom{a}{n} T^n
\end{equation}
is the Dirac measure at $a$, that is, the measure $\mu_a$ such that
\begin{equation*}
\int f \, d\mu_a = f(a)
\end{equation*}
for all $f \in \mathcal{C}(\mathbb{Z}_p, \mathbb{C}_p)$. \qed
\end{lemma}

\subsection{The interpolation principle}

Let $\Delta$ be the linear operator $(1+T) \dfrac{d}{dT}$ acting on $\mathbb{C}_p[[T]]^{bd}$. Let $\mu \in \mathcal{M}(\mathbb{Z}_p, \mathbb{C}_p)$ be a measure. We have
\begin{align*}
\Delta F_\mu(T) & = (1+T) \sum_{n \geq 1} n \int \binom{x}{n} \, d\mu \cdot T^{n-1} = \sum_{n \geq 0} \int \left[(n+1) \binom{x}{n+1} + n \binom{x}{n} \right] d\mu \cdot T^n \\
& = \sum_{n \geq 0} \int x \binom{x}{n} d\mu \cdot T^n = F_{x\mu}(T).
\end{align*}
Thus we have proved the first part of the result below; the second follows from Remark~\ref{rem:uniqfmu}.
\begin{lemma}\label{lem:intdeltak}
Let $\mu$ be a measure with associated power series $F_\mu$. Then the measure associated with the power series $\Delta F_\mu$ is $x \mu$. In particular,
\begin{equation*}
\Delta^k F_\mu(T)_{|T=0} = \int x^k \, d\mu
\end{equation*} 
for any integer $k \geq 0$. \qed
\end{lemma}

We can now state the main result of this section. 
\begin{theorem}\label{th:procinterpol}
Let $(a_n)_{n \geq 0}$ be a sequence of elements of $\mathbb{C}_p$. Assume there exists a power series $F \in \mathbb{C}_p[[T]]^{bd}$ such that for all $k \geq 0$ we have
\begin{equation*}
\Delta^k F(T)_{|T = 0} = a_k
\end{equation*}
and that the associated measure $\mu_F$ has support in $1 + q\mathbb{Z}_p$. Let $f : \mathbb{Z}_p \to \mathbb{C}_p$ be defined by
\begin{equation*}
f(s) := \int \phi_s(x) \, d\mu_F(x).
\end{equation*}
Then $f$ is a continuous function such that
\begin{equation}
f(k) = a_k
\end{equation}
for all $k \in \mathbb{N}$.
\end{theorem}
\begin{proof}
It is clear from Lemmas~\ref{lem:phicont} and \ref{lem:intcont} that $f$ is continuous. For $k \in \mathbb{N}$ we compute
\begin{align*}
f(k) & = \int_{1 + q\mathbb{Z}_p} \phi_k(x) \, d\mu_F(x) && \text{since $\mu$ has support in $1+q\mathbb{Z}_p$} \\
& =  \int_{1+_q\mathbb{Z}_p} x^k \, d\mu_F && \text{by Lemma~\ref{lem:expfact}} \\
& =  \int x^k \, d\mu_F && \text{since $\mu$ has support in $1+q\mathbb{Z}_p$} \\
& = \Delta^k F(T)_{|T=0} && \text{by Lemma~\ref{lem:intdeltak}} \\[1\jot]
& = a_k. && \qedhere
\end{align*}
\end{proof}

\section{Values of zeta functions at negative integers}

Let $E$ be a totally real number field of degree $d$ with ring of integers $\mathbb{Z}_E$. We consider $E$, and all other number fields, as subfields of the algebraic closure $\bar{\mathbb{Q}}$ of $\mathbb{Q}$ contained in $\mathbb{C}$. We also fix once and for all an embedding of $\bar{\mathbb{Q}}$ into $\mathbb{C}_p$. For $\alpha$ in $E$, we denote by $\alpha^{(i)} \in \mathbb{R}$, $i=1,\dots,d$, its conjugates. An element $\alpha \in E$ is \emph{totally positive} if $\alpha^{(i)} > 0$ for $i=0,\dots,d$. We write $\alpha \gg 0$. The subgroup of totally positive numbers in $E^\times$ is denoted $E^+$ and we let $\mathbb{Z}_E^+ := E^+ \cap \mathbb{Z}_E$ be the set of totally positive algebraic integers in $E$. Let $\mathcal{N} = \mathcal{N}_{E/\mathbb{Q}}$ denote the absolute norm of the group $I(E)$ of ideals of $E$. By abuse, for $\alpha$ a non-zero element in $E$ we write $\mathcal{N}(\alpha) := \mathcal{N}(\alpha\mathbb{Z}_E)$. When $\alpha$ is totally positive $\mathcal{N}(\alpha)$ equals the absolute norm of~$\alpha$.

Let $\mathfrak{m} := \mathfrak{f} \mathfrak{z}$ be a \emph{modulus} of $E$, that is, the formal product of an integral ideal $\mathfrak{f}$ of $E$ (the finite part) and a subset $\mathfrak{z}$ of the set of infinite places of $E$ (the infinite part). We use the notations
\begin{itemize}
\item[]
\begin{itemize}
\item[$E_\mathfrak{m}$]
for the subgroup of elements of $E$ that are congruent (multiplicatively) \\ to $1$ modulo $\mathfrak{m}$,
\smallskip
\item[$I_{\mathfrak{m}}(E)$]
for the subgroup of fractional ideals of $E$ that are relatively prime 
%%% with 
to
$\mathfrak{f}$, 
\smallskip
\item[$\Cl_\mathfrak{m}(E)$]
for the \emph{ray class group of $E$ modulo $\mathfrak{m}$},
that is, the quotient of $I_{\mathfrak{m}}(E)$ by \\ the subgroup of principal ideals generated by elements of $E_\mathfrak{m}$, and
\smallskip
\item[$h_\mathfrak{m}(E)$]
for the cardinality of $\Cl_\mathfrak{m}(E)$.
\smallskip
\end{itemize}
\end{itemize}
Finally, we set $U_\mathfrak{m}(E):= U(E) \cap E_\mathfrak{m}$ where $U(E)$ is the unit group of $E$.

\subsection{Twisted partial zeta functions}\label{subsec:tzf} 

Let $\mathfrak{a}$ be a fractional ideal of $E$, relatively prime to $\mathfrak{f}$. The \emph{partial zeta function} is defined, for $s \in \mathbb{C}$ with $\Re(s) > 1$, by 
\begin{equation*}
Z_{\mathfrak{m}}(\mathfrak{a}^{-1}; s) := \sum_{\mathfrak{b} \sim_\mathfrak{m} \mathfrak{a}^{-1}} \mathcal{N}(\mathfrak{b})^{-s} 
\end{equation*}
where the sum is over all the integral ideals $\mathfrak{b}$ that are in the same class of $\Cl_{\mathfrak{m}}(E)$ as the inverse of $\mathfrak{a}$. For $\mathfrak{c}$, an ideal of $E$, relatively prime with $\mathfrak{f}$, the \emph{twisted partial zeta function} is defined, for $\Re(s) > 1$, by 
\begin{equation}\label{eq:deftz}
Z_{\mathfrak{m}}(\mathfrak{a}^{-1}, \mathfrak{c}; s) := \mathcal{N}(\mathfrak{c})^{1-s} Z_\mathfrak{m}((\mathfrak{ac})^{-1}; s) - Z_\mathfrak{m}(\mathfrak{a}^{-1}; s). 
\end{equation}

The partial zeta functions have meromorphic continuation to the complex plane with a simple pole at $s = 1$. Since the partial zeta functions all have the same residue at $s = 1$, they cancel out in \eqref{eq:deftz} and the twisted partial zeta functions have analytic continuation to the whole complex plane.

Let $\chi$ be a character on the ray class group $\Cl_\mathfrak{m}(E)$, and let $L_{\mathfrak{m}}(\chi; s)$ be the corresponding Hecke $L$-function defined, for $\Re(s) > 1$, by
\begin{equation}\label{eq:defLcplx}
L_\mathfrak{m}(\chi; s) := \prod_{\mathfrak{q} \notdiv \mathfrak{m}} \left(1 - \chi(\mathfrak{q}) \mathcal{N}(\mathfrak{q})^{-s}\right)^{-1}
\end{equation}
where the product is over all the prime ideals of $E$ not dividing the finite part $\mathfrak{f}$ of the modulus $\mathfrak{m}$. The link between Hecke $L$-functions and partial zeta function provides a way to express the former in terms of twisted partial zeta functions. 

\begin{proposition}\label{prop:lfunc} 
For all $s \in \mathbb{C}$ we have
\begin{equation*}
\left(\chi(\mathfrak{c})\mathcal{N}(\mathfrak{c})^{1-s} - 1\right) L_\mathfrak{m}(\chi; s) = \sum_{i=1}^{h_\mathfrak{m}(E)} \bar\chi(\mathfrak{a}_i) Z_{\mathfrak{m}}(\mathfrak{a}_i^{-1}, \mathfrak{c}; s)
\end{equation*}
where the sum is over ideals $\mathfrak{a}_i$ representing all the classes of $\Cl_\mathfrak{m}(E)$. 
\end{proposition}

\begin{proof}
We have
\begin{align*}
\sum_{i=1}^{h_\mathfrak{m}(E)} \bar\chi(\mathfrak{a}_i) Z_{\mathfrak{m}}(\mathfrak{a}_i^{-1}, \mathfrak{c}; s) 
& = \mathcal{N}(\mathfrak{c})^{1-s}\sum_{i=1}^{h_\mathfrak{m}(E)} \bar\chi(\mathfrak{a}_i) Z_{\mathfrak{m}}((\mathfrak{a_ic})^{-1}; s) - \sum_{i=1}^{h_\mathfrak{m}(E)} \bar\chi(\mathfrak{a}_i) Z_{\mathfrak{m}}(\mathfrak{a}_i^{-1}; s) \\
& = \left(\chi(\mathfrak{c})\mathcal{N}(\mathfrak{c})^{1-s} - 1\right) \sum_{i=1}^{h_\mathfrak{m}(E)} \bar\chi(\mathfrak{a}_i) Z_{\mathfrak{m}}(\mathfrak{a}_i^{-1}; s) \\
& =  \left(\chi(\mathfrak{c})\mathcal{N}(\mathfrak{c})^{1-s} - 1\right) L_\mathfrak{m}(\chi; s). \qedhere
\end{align*}
\end{proof}

\medskip

We now make some important additional hypotheses. 
\begin{hyp}\  
\begin{enumerate}
\smallskip
\item[(H1)] the finite part $\mathfrak{f}$ of the modulus $\mathfrak{m}$ is divisible by $q$;
\smallskip
\item[(H2)] the infinite part $\mathfrak{z}$ of the modulus $\mathfrak{m}$ contains all the infinite (real) places of $E$;
\smallskip
\item[(H3)] $\mathfrak{c}$ is a prime ideal of residual degree $1$.
\end{enumerate}
\end{hyp}

We will denote by $c$ the prime number below $\mathfrak{c}$; therefore $c = \mathcal{N}(\mathfrak{c})$ by (H3). 

\begin{remark}
If $\mathfrak{m}$ does not satisfy both (H1) and (H2) we can enlarge the modulus so that it does satisfy these conditions and we can lift $\chi$ to a character of the new modulus. Adding all the infinite places to $\mathfrak{z}$ to satisfy (H2) does not actually change the $L$-function. Replacing $\mathfrak{f}$ by the lcm of $\mathfrak{f}$ and $q$ to satisfy (H1) has the effect of removing the Euler factors of prime ideals above $p$ in \eqref{eq:defLcplx}. This is necessary to be able to do the $\p$-adic interpolation. Another way to achieve this would be to drop (H1) and to require instead that only the elements coprime to $q$ are kept in the cone decompositions (see Subsection~\ref{subsec:conedec}). From a computational point of view these two possibilities are basically the same. 
\end{remark}

\begin{remark}
The construction of Cassou-Noguès \cite{cn:zeta} additionally requires $\mathfrak{c}$ to be relatively prime to the co-different of $E$. However, as we will see in the next subsection, this is not actually necessary. 
\end{remark}

\subsectb{Additive characters modulo $\mathfrak{c}$}{Additive characters modulo c}

An \emph{additive character} modulo $\mathfrak{c}$ is a group homomorphism $\xi$ from the additive group $\mathbb{Z}_E$ to the multiplicative group $\mathbb{C}^\times$ whose kernel contains $\mathfrak{c}$. We denote by $X(\mathfrak{c})$ the set of all these characters.

\begin{lemma}\label{lem:charc}
$X(\mathfrak{c})$ is a finite group of order $c$ and all elements in $X(\mathfrak{c})$ but the trivial character have kernel $\mathfrak{c}$. Furthermore, for $x \in \mathbb{Z}_E$ we have
\begin{equation*}
\sum_{\xi \in X(\mathfrak{c})} \xi(x) =
\begin{cases}
\,c & \text{ if $x \in \mathfrak{c}$}, \\
\,0 & \text{ otherwise}.
\end{cases}
\end{equation*}
\end{lemma}

\begin{proof}
Let $\chi$ be a non-trivial character in $X(\mathfrak{c})$. Then $\mathbb{Z}_E/\Ker(\chi)$ is a quotient group of $\mathbb{Z}_E/\mathfrak{c} \cong \mathbb{Z}/c\mathbb{Z}$ and thus $\xi$ is completely determined by its value on $1$, which can be an arbitrary non-trivial $c$-th root of unity. Therefore there are $c-1$ non-trivial characters and each has kernel $\mathfrak{c}$ since the non-trivial $c$-th roots of unity all have order $c$. The last statement is the classical orthogonality relation for characters.
\end{proof}

\begin{proposition}\label{prop:1}
Let $\mathfrak{a}$ be an integral ideal coprime to $\mathfrak{fc}$. Let $A$ be a set of representatives of the elements of $\mathfrak{a} \cap E_\mathfrak{m}$ under the {\rm(}multiplicative{\rm)} action of $U_{\mathfrak{m}}(E)$. Then for $\Re(s) > 1$
\begin{equation}\label{eq:defzxi}
Z_{\mathfrak{m}}(\mathfrak{a}^{-1}, \mathfrak{c}; s)
= \mathcal{N}(\mathfrak{a})^s \sum_{\substack{\xi \in X(\mathfrak{c}) \\ \xi \ne 1}}\,\sum_{\alpha \in A} \xi(\alpha) \mathcal{N}(\alpha)^{-s}.  
\end{equation}
\end{proposition}

\begin{proof}
An ideal $\mathfrak{b}$ is equivalent to $\mathfrak{a}^{-1}$ modulo $\mathfrak{m}$ if and only if there exists $\alpha \in E_\mathfrak{m}$ such that $\mathfrak{b} = \alpha\mathfrak{a}^{-1}$. Furthermore, $\mathfrak{b}$ is integral if and only if $\alpha$ belongs to $\mathfrak{a} \cap E_\mathfrak{m}$, and two elements $\alpha$ and $\alpha'$ of $\mathfrak{a} \cap E_\mathfrak{m} $ yield the same ideal $\alpha\mathfrak{a}^{-1} = \alpha'\mathfrak{a}^{-1}$ if and only if there exists a unit $\epsilon \in U_\mathfrak{m}$ such that $\alpha = \epsilon\alpha'$. Therefore there is a one-to-one correspondence between the integral ideals equivalent to $\mathfrak{a}^{-1}$ modulo $\mathfrak{m}$ and the elements of $A$. Thus we have
\begin{equation}\label{eq:1}
Z_{\mathfrak{m}}(\mathfrak{a}^{-1}; s) 
    = \sum_{\alpha \in A}\mathcal{N}(\alpha\mathfrak{a}^{-1})^{-s} 
    = \mathcal{N}(\mathfrak{a})^s \sum_{\alpha \in A}\mathcal{N}(\alpha)^{-s}.
\end{equation}
We now compute
\begin{align*}
Z_{\mathfrak{m}}(\mathfrak{a}^{-1}, \mathfrak{c}; s) & = \mathcal{N}(\mathfrak{c})^{1-s} 
Z_{\mathfrak{m}}((\mathfrak{a}\mathfrak{c})^{-1}; s) - Z_{\mathfrak{m}}(\mathfrak{a}^{-1}; s) \\
& = \mathcal{N}(\mathfrak{a})^s
\biggl(
\mathcal{N}(\mathfrak{c}) \sum_{\alpha \in A \cap \mathfrak{c}} \mathcal{N}(\alpha)^{-s} - \sum_{\alpha \in A} \mathcal{N}(\alpha)^{-s}
\biggr)
\end{align*}
using \eqref{eq:1} and the fact that $A \cap \mathfrak{c}$ is a set of representatives of $\mathfrak{ac} \cap E_\mathfrak{m}$ modulo $U_\mathfrak{m}(E)$, since $\mathfrak{a}$ and $\mathfrak{c}$ are coprime. Finally, we obtain the conclusion using Lemma~\ref{lem:charc}.
\end{proof}

\subsection{Cone decomposition}\label{subsec:conedec} 

Let $\beta, \lambda_1, \dots, \lambda_g$ be elements in $\mathfrak{a} \cap \mathbb{Z}_E^+$, with $1 \leq g \leq d$, such that the $\lambda_i$'s are linearly independent. We define the \emph{discrete cone\footnote{We will say simply \emph{cone}.} with base point $\beta$ and generators $\lambda_1, \dots, \lambda_g$}
as the following subset of~$\mathfrak{a} \cap \mathbb{Z}_E^+\,$:
\begin{equation}\label{eq:defcone}
C(\beta; \lambda_1, \dots, \lambda_g) := \biggl\{\beta + \sum_{i=1}^g n_i \lambda_i \text{ with } n_i \in \mathbb{N} \text{ for } 1 \leq i \leq g \biggr\}.
\end{equation}
Following the work of Shintani \cite{shintani}, we have the following result of Pi.~Cassou-Noguès.

\begin{theorem}[Cassou-Noguès]\label{th:conedec}
There exists a finite family $\{C_1, \dots, C_m\}$ of disjoint discrete cones of $E$ with base points belonging to $\mathfrak{a} \cap E_\mathfrak{m}$ and generators belonging to $(\mathfrak{af} \cap \mathbb{Z}_E^+)$ such that a set of representatives of $\mathfrak{a} \cap E_\mathfrak{m}$ under the action of $U_\mathfrak{m}(E)$ is given by the union $C_1 \cup \cdots \cup C_m$.
\end{theorem}
\begin{proof}
This is essentially \cite[Lemma~1]{cn:zeta}.
\end{proof}
 
A finite set $\{C_1, \dots, C_m\}$ of cones satisfying Theorem~\ref{th:conedec} is called a \emph{cone decomposition of $\mathfrak{a}$ modulo $\mathfrak{m}$}. A cone $C$ is \emph{$\mathfrak{c}$-admissible} if none of its generators belong to $\mathfrak{c}$. A cone decomposition $\{C_1, \dots, C_m\}$ is $\mathfrak{c}$-admissible if all the cones $C_i$ are $\mathfrak{c}$-admissible (see Remark~\ref{rk:choicec} on the existence of such a decomposition). From Proposition~\ref{prop:1}, we have, for $\Re(s) > 1$,
\begin{equation}\label{eq:zetacones}
Z_{\mathfrak{m}}(\mathfrak{a}^{-1}, \mathfrak{c}; s) = \mathcal{N}(\mathfrak{a})^s \sum_{\substack{\xi \in X(\mathfrak{c}) \\ \xi \ne 1}} \sum_{j = 1}^m \sum_{\alpha \in C_j} \xi(\alpha) \mathcal{N}(\alpha)^{-s}.  
\end{equation}

\subsection{Cone zeta functions}\label{subsec:zetcone}

Let $C := C(\beta; \lambda_1, \dots, \lambda_g)$ be a $\mathfrak{c}$-admissible cone and $\xi$ be a non-trivial element of $X(\mathfrak{c})$. We define the \emph{zeta function of the pair} $(C, \xi)$, for $\Re(s) > 1$, by
\begin{equation}\label{eq:defzcxi}
Z(C, \xi; s) := \sum_{\alpha \in C} \xi(\alpha) \mathcal{N}(\alpha)^{-s}.
\end{equation}
We associate with the same data a power series $F(C, \xi; T_1, \dots, T_d)$ in $\bar{\mathbb{Q}}[[\underline{T}]] := \bar{\mathbb{Q}}[[T_1, \dots, T_d]]$ in the following way. First, for $r \in \mathbb{C}$, we define a power series in $\mathbb{C}[[T]]$ by\footnote{This, of course, gives the usual definition when $r \in \mathbb{N}$.} 
\begin{equation*}
(1 + T)^r := \sum_{n \geq 0} \binom{r}{n} T^n.
\end{equation*}
Then, for $\alpha \in E$, we define the following power series in $\bar{\mathbb{Q}}[[\underline{T}]]$ 
\begin{equation}
(1 + \underline{T})^\alpha := \prod_{i=1}^d (1 + T_i)^{\alpha^{(i)}}. 
\end{equation}
And finally, we set
\begin{equation}\label{eq:defGcxi}
G(C, \xi; \underline{T}) := \frac{\xi(\beta) (1+\underline{T})^\beta}{\prod\limits_{i=1}^g \left(1 - \xi(\lambda_i)(1 + \underline{T})^{\lambda_i}\right)}.
\end{equation}

\begin{remark}\label{rem:cstterm}
From Lemma~\ref{lem:charc} and the fact that $C$ is $\mathfrak{c}$-admissible it follows that $\xi(\lambda_i)$ is a non-trivial $c$-th root of unity for all $i$'s and thus the constant term of the denominator is non-zero. Therefore $G(C, \xi; \underline{T})$ is indeed a power series. In fact, its constant term is an algebraic integer divisible only by primes above $c$.
\end{remark}

The \emph{cone zeta function} of $C$ (twisted by $\mathfrak{c}$) is defined by 
\begin{equation}\label{eq:defzetac}
Z(C, \mathfrak{c}; s) := \sum_{\substack{\xi \in X(\mathfrak{c}) \\ \xi \ne 1}} Z(C, \xi; s),
\end{equation}
and in a similar way
\begin{equation}\label{eq:defGc}
G(C, \mathfrak{c}; T) := \sum_{\substack{\xi \in X(\mathfrak{c}) \\ \xi \ne 1}} G(C, \xi; T).
\end{equation}
Finally, we define the \emph{$\underline{\Delta}$-operator} acting on $\bar{\mathbb{Q}}[[\underline{T}]]$ by 
\begin{equation}\label{eq:delta}
\underline{\Delta} := \prod_{i=1}^d (1 + T_i) \frac{\partial}{\partial T_i}.
\end{equation}

\begin{theorem}[Shintani]\label{th:interpol}
The function $Z(C, \mathfrak{c}; s)$ admits an analytic continuation to $\mathbb{C}$, and, for any integer $k \geq 0$, we have
\begin{equation}
Z(C, \mathfrak{c}; -k) = \underline{\Delta}^k G(C, \mathfrak{c}; \underline{T})_{|\underline{T}=\underline{0}}.
\end{equation}
\end{theorem}
\begin{proof}
We use the following lemma from Colmez \cite[Lemma~3.2]{colmez:residue}.
\begin{lemma}
For $z_1, \dots, z_d \in \mathbb{R}^+$ let $f(z_1, \dots, z_d)$ be a $C^\infty$-function such that it and all its derivatives tend to $0$ rapidly at infinity. For all $(s_1, \dots, s_d) \in \mathbb{C}^d$ such that $\Re(s_i) > 0$ for $i = 1, \dots, d$ define the function 
\begin{equation}
M(f; s_1, \dots, s_d) := \int_{(\mathbb{R}^{+*})^d} f(z_1, \dots, z_d)\,\frac{z_1^{s_1} \cdots z_d^{s_d}}{\Gamma(s_1) \cdots \Gamma(s_d)}\,\frac{dz_1}{z_1} \cdots \frac{dz_d}{z_d}.
\end{equation}
Then $M(f; \cdot)$ admits an analytic continuation to $\mathbb{C}^d$ and satisfies
\begin{equation*}
M(f; -k_1, \dots, -k_d) = \prod_{i=1}^d \left(-\frac{\partial}{\partial z_i}\right)^{k_i} f(z_1, \dots, z_d)_{|z_1 = \dots = z_d = 0}
\end{equation*}
for all $(k_1, \dots, k_d) \in \mathbb{N}^d$. \qed
\end{lemma}
Let $\xi \in X(\mathfrak{c})$ with $\xi \ne 1$. We define a function $f_\xi$ by
\begin{equation}
f_\xi(z_1, \dots, z_d) := G(C, \xi; e^{-z_1} - 1, \dots, e^{-z_d} - 1) = \frac{\xi(\beta) e^{- T_{\underline{z}}(\beta)}}{\prod\limits_{i=1}^g \big(1 - \xi(\lambda_i) e^{-T_{\underline{z}}(\lambda_i)}\big)}
\end{equation}
where, for $\alpha \in \mathbb{Z}_E$ and $\underline{z} := (z_1, \dots, z_d) \in \mathbb{C}^d$, we set $T_{\underline{z}}(\alpha) := z_1\alpha^{(1)} + \dots + z_d\alpha^{(d)}$. It is clear that this function satisfies the hypothesis of the lemma. Furthermore, for $\underline{z} \in (\mathbb{R}^{+*})^n$ and $\alpha \gg 0$ we have $0 < e^{-T_{\underline{z}}(\alpha)} < 1$, and therefore, by expanding the numerator and by \eqref{eq:defcone},
\begin{equation*}
f_\xi(z_1, \dots, z_d) = \sum_{\underline{n} \in \mathbb{N}^g} \xi(\beta + n_1\lambda_1 + \cdots + n_g\lambda_g) e^{-T_{\underline{z}}(\beta + n_1\lambda_1 + \cdots + n_g\lambda_g)} = \sum_{\alpha \in C} \xi(\alpha) e^{-T_{\underline{z}}(\alpha)}.
\end{equation*}
Thus we find that
\begin{equation*}
M(f_\xi; s_1, \dots, s_d) = \sum_{\alpha \in C} \xi(\alpha) \int_{\mathbb{R}^{+*}} e^{-z_1\alpha^{(1)}} \frac{z_1^{s_1}}{\Gamma(s_1)} \frac{dz_1}{z_1}\;\cdots \int_{\mathbb{R}^{+*}} e^{-z_d\alpha^{(d)}} \frac{z_d^{s_d}}{\Gamma(s_d)} \frac{dz_d}{z_d}.
\end{equation*}
Since 
\begin{equation*}
\int_{\mathbb{R}^{+*}} e^{-z a} z^{s} \frac{dz}{z} = \Gamma(s) \, a^{-s}
\end{equation*}
for $a \in \mathbb{R}^{+*}$, it follows that
\begin{equation*}
M(f_\xi; s, \dots, s) = \sum_{\alpha \in C} \xi(\alpha) (\alpha^{(1)})^s \cdots (\alpha^{(d)})^s = Z(C, \xi; s)
\end{equation*}
for $\Re(s) > 1$.
Lastly, we find that
\begin{equation*}
-\frac{\partial}{\partial z_i} f_\xi(z_1, \dots, z_d) = \Bigl((1+T_i) \frac{\partial}{\partial T_i} G\Bigr) (C, \xi; e^{-z_1} - 1, \dots, e^{-z_d} - 1)
\end{equation*}
for all $i$'s. The conclusion now follows from the lemma, \eqref{eq:defzetac}, and \eqref{eq:defGc}.
\end{proof}

Recall that we have embedded $\bar{\mathbb{Q}}$ into $\mathbb{C}_p$. Thus we can see $G(C, \mathfrak{c}; \underline{T})$ as having coefficients in $\mathbb{C}_p$ and, generalizing what we did in the first part to higher dimensions, we can try to interpret $G(C, \mathfrak{c}; \underline{T})$ as a measure over $\mathbb{Z}_p^d$. There are two problems. First, the power series $G(C, \mathfrak{c}; \underline{T})$ having several variables complicates things, at least from a computational point of view; it would be much easier to deal with a \emph{one-variable} power series.\footnote{This is not a big problem though; in fact the computations done in \cite{RS1} use two-variable power series.} Second, this power series might not have bounded coefficients, since $\binom{\beta}{n}$ has arbitrarily large $\p$-adic absolute value when $\beta$ is not a rational $\p$-adic integer. This problem
can be solved by making a change of variable in $G(C, \mathfrak{c}; \underline{T})$, as in \cite{solomon:cocycle2}, to transform it into a power series with bounded coefficients. In the next subsection we will see how to transform the power series $G(C, \mathfrak{c}; \underline{T})$ into a one-variable power series satisfying the direct analog of Theorem~\ref{th:interpol}. It will turn out that this power series has $\p$-adic integral coefficients, thus solving both problems at the same time. 

\subsectb{The $\Omega$ operator}{The Omega operator}\label{subsec:omega}

We now explain how to construct an operator that sends $G(C, \mathfrak{c}; \underline{T})$ to an one-variable power series satisfying properties analogous to that of Theorem~\ref{th:interpol}. From the definition of $\underline{\Delta}$, we see that $\Omega$ should satisfy 
\begin{equation*}
\Omega((1+T_1)^{a_1} \cdots (1+T_d)^{a_d}) = (1+T)^{a_1 \cdots a_d}.
\end{equation*}
Writing $T_i^{a_i} = ((1+T_i)-1)^{a_i}$ and developing, we get the following formal definition. Let $\Omega$ be the linear function from $\mathbb{C}[\underline{T}]$ to $\mathbb{C}[T]$ defined for $(a_1, \dots, a_d) \in \mathbb{N}^d$ by
\begin{equation*}
\Omega(T_1^{a_1} \cdots T_d^{a_d}) := (-1)^{a_1 + \cdots + a_d} \sum_{n_1=0}^{a_1} \cdots \sum_{n_d=0}^{a_d} \biggl(\,\prod_{i=1}^d (-1)^{n_i} \binom{a_i}{n_i}\!\biggr) (1 + T)^{n_1 \cdots n_d}.
\end{equation*}
The lemma below establishes that this application can be continuously extended to an application from $\mathbb{C}[[\underline{T}]]$ to $\mathbb{C}[[T]]$.
\begin{lemma}\label{lem:omval}
Let $(a_1, \dots, a_d) \in \mathbb{N}^d$. Then, $\Omega(T_1^{a_1} \cdots T_d^{a_d})$ is divisible by $T^{\max(a_1, \dots, a_d)}$.
\end{lemma}
\begin{proof}
Assume without loss of generality that $a_d$ is the largest of the $a_i$'s. We have
\begin{align*}
\Omega(T_1^{a_1} \cdots T_d^{a_d}) & =
(-1)^{a_1 + \cdots + a_{d-1}} \sum_{n_1=0}^{a_1}\,\cdots\!\sum_{n_{d-1}=0}^{a_{d-1}} (-1)^{n_1 + \cdots + n_{d-1}} \binom{a_1}{n_1} \cdots\binom{a_{d-1}}{n_{d-1}} \\[-1\jot]
& \omit \hfill $\displaystyle {}\times \sum_{n_d=0}^{a_d} (-1)^{a_d - n_d} \binom{a_d}{n_d} \big((1 + T)^{n_1 \cdots n_{d-1}}\big)^{n_d}$\ignorespaces \\[+1\jot]
& = (-1)^{a_1 + \cdots + a_{d-1}} \sum_{n_1=0}^{a_1}\,\cdots\!\sum_{n_{d-1}=0}^{a_{d-1}} \biggl(\,\prod_{i=1}^{d-1} (-1)^{n_i} \binom{a_i}{n_i}\!\biggr) \big((1 + T)^{n_1 \cdots n_{d-1}} - 1\big)^{a_d}
\end{align*}
and every term in this sum is divisible by $T^{a_d}$.
\end{proof}
We now prove the main property of the $\Omega$ operator, that is, that it ``commutes'' with the operator $\underline{\Delta}$.

\begin{proposition}\label{prop:odcommute}
For $A \in \mathbb{C}[[\underline{T}]]$ we have
\begin{equation*}
\Omega(\underline{\Delta} (A)) = \Delta(\Omega(A)).
\end{equation*}
\end{proposition}
\begin{proof}
By linearity and the fact that the operators $\Delta$, $\underline{\Delta}$, and $\Omega$ are linear and continuous,\footnote{We leave to the reader the verification that $\Delta$ and $\underline{\Delta}$ are continuous.} it is enough to prove the result for monomials $T_1^{a_1} \cdots T_d^{a_d}$ with $(a_1, \dots, a_d) \in \mathbb{N}^d$. But any such monomial can be written as a finite linear combination of $(1+T_1)^{b_1} \cdots (1+T_d)^{b_d}$ with $(b_1, \dots, b_d)^d \in \mathbb{N}$, for which the result is direct by construction.
\end{proof}

Let $C$ be a $\mathfrak{c}$-admissible cone. We set
\begin{equation*}
F(C, \mathfrak{c}; T) := \Omega(G(C, \mathfrak{c}; \underline{T})).
\end{equation*} 
Using Proposition~\ref{prop:odcommute}, the next result is a direct consequence of Theorem~\ref{th:interpol}.
\begin{theorem}
For any integer $k \geq 0$ we have
\begin{equation*}
Z(C, \mathfrak{c}; -k) = \Delta^k F(C, \mathfrak{c}; T)_{|T=0}. \tag*{\hspace{-1em}\qed}
\end{equation*}
\end{theorem}

Let $\mathcal{D} := \{C_1, \dots, C_m\}$ be a $\mathfrak{c}$-admissible cone decomposition of $\mathfrak{a}$ modulo $\mathfrak{m}$. We define
\begin{equation}\label{eq:defF}
F_\mathfrak{m}(\mathfrak{a}, \mathfrak{c}; T) := \sum_{j=1}^m F(C_j, \mathfrak{c}; T).
\end{equation}

\begin{corollary}\label{cor:interpolF}
For any integer $k \geq 0$ we have
\begin{equation}\label{eq:interpolF}
Z_{\mathfrak{m}}(\mathfrak{a}^{-1}, \mathfrak{c}; -k) = \mathcal{N}(\mathfrak{a})^{-k} \Delta^k F_{\mathfrak{m}}(\mathfrak{a}, \mathfrak{c}; T)_{|T=0}.
\end{equation}
\end{corollary}
\begin{proof}
Clear from \eqref{eq:zetacones}.
\end{proof}

To conclude this subsection, we prove that the power series $F_{\mathfrak{m}}(\mathfrak{a}, \mathfrak{c}; T)$ does not depend on the choice of the cone decomposition $\mathcal{D}$. Indeed, the previous corollary prescribes the values of $\Delta^k F_{\mathfrak{m}}(\mathfrak{a}, \mathfrak{c}; T)_{|T=0}$ for all $k \geq 0$ which, using the following result, ensures the unicity of $F_{\mathfrak{m}}(\mathfrak{a}, \mathfrak{c}; T)$.

\begin{lemma}
Let $F(T) \in \mathbb{C}[[T]]$ and let $k \geq 0$ be an integer such that $T^k \div \Delta F(T)$. Then $T^{k+1} \div F(T) - F(0)$. Furthermore, if $\Delta^k F(T)_{|T=0} = 0$ for all $k \geq 0$ then $F(T) = 0$.
\end{lemma}

\begin{proof}
Write $F(T) := \sum\limits_{n \geq 0} f_n T^n$. We compute
\begin{equation*}
\Delta F(T) = f_1 + \sum_{n \geq 1} \big((n+1)f_{n+1} + nf_n\big) T^n.
\end{equation*}
From this it is easy to see that $T^k \div \Delta F(T)$ implies $f_1 = \cdots = f_{k+1} = 0$, which proves the first assertion. For the second, let $k \geq 1$. Since $\Delta^k F(T)_{|T=0} = 0$, that is, $T \div \Delta (\Delta^{k-1} F(T))$, we have $T^2 \div \Delta^{k-1} F(T)$ by the first part, as $\Delta^{k-1} F(T)_{|T=0}$ by hypothesis. Repeating this process (and using the fact that $F(0) = \Delta^0 F(T)_{|T=0} = 0$), we eventually get $T^{k+1} \div F(T)$. Since $k$ is arbitrary, it follows that $F(T) = 0$.
\end{proof}

\sectm{$\p$-adic $L$-functions}{p-adic L-functions}\label{sec:interpolation} 

We put together the results of the last two sections to construct the $\p$-adic $L$-functions. 

\subsectb{Some properties of $F_{\mathfrak{m}}(\mathfrak{a}, \mathfrak{c}; T)$}{Some properties of Fm[a,c;T]}\label{subsec:propF}

Let $\mathfrak{a}$ be an integral ideal coprime to $\mathfrak{c}$ and $\mathfrak{m}$. We prove that the power series $F_{\mathfrak{m}}(\mathfrak{a}, \mathfrak{c}; T)$ possesses the properties required to apply Theorem~\ref{th:procinterpol}. We start by proving a useful expression for $F(C, \xi; T)$ modulo powers of $T$.

\begin{theorem}\label{th:main}
For integers $k$ and $K$ with $0 \leq k \leq K$ define the rational function
\begin{equation}\label{eq:defB}
B_{k,K}(x) := (-1)^{k} \sum_{n=k}^K \binom{n}{k} \biggl(\frac{x}{x-1}\biggr)^{\!n} \in \mathbb{Q}(x).
\end{equation}
Let $C := C(\beta; \lambda_1, \dots, \lambda_g)$ be a $\mathfrak{c}$-admissible cone and let $\xi$ be a non-trivial element of $X(\mathfrak{c})$. For $N \geq 0$ define the polynomial $F_N(C, \xi; T) \in \mathbb{Q}(\xi)[T]$ by 
\begin{equation*}
F_N(C, \xi; T) := 
A(C, \xi)\!\sum_{k_1, \dots, k_g = 0}^{(N-1)d} (1 + T)^{\mathcal{N}(\beta + \underline{k} \cdot \underline{\lambda})}\,\prod_{i=1}^g B_{k_i, (N-1)d}(\xi(\lambda_i))
\end{equation*}
where $\underline{k} \cdot \underline{\lambda} := k_1 \lambda_1 + \cdots + k_g \lambda_g \in \mathbb{Z}_E$ and
\begin{equation*}
A(C, \xi) := \frac{\xi(\beta)}{\prod\limits_{i=1}^g (1 - \xi(\lambda_i))}.
\end{equation*} 
Then
\begin{equation*}
F(C, \xi; T) \equiv F_N(C, \xi; T) \pmod{T^N}.
\end{equation*}
\end{theorem}

\begin{proof}
To simplify the notation we write $a_i := \xi(\lambda_i)/(1-\xi(\lambda_i))$ and $A := A(C, \xi)$. We compute
\begin{align*}
G(C, \xi; \underline{T})
& = \frac{\xi(\beta) (1+\underline{T})^{\beta}}{\prod\limits_{i=1}^g \left(1 - \xi(\lambda_i)(1 + \underline{T})^{\lambda_i}\right)} = A \, \frac{(1+\underline{T})^{\beta}}{\prod\limits_{i=1}^g \left(1 - a_i \left((1 + \underline{T})^{\lambda_i}-1\right)\right)} \\
& = A \, (1+\underline{T})^{\beta}\!\sum_{n_1, \dots, n_g \geq 0}\,\prod_{i=1}^g a_i^{n_i} \left((1 + \underline{T})^{\lambda_i}-1\right)^{n_i}.
\end{align*}
Let $\mathcal{I}$ be the ideal of $\mathbb{C}_p[[\underline{T}]]$ generated by the monomials $T_1^{a_1} \cdots\,T_d^{a_d}$ where $\max(a_1, \dots, a_d) \geq N$. For any $P(\underline{T}) \in \mathcal{I}$, it follows by Lemma~\ref{lem:omval} that $\Omega(P) \in T^N \mathbb{C}_p[[T]]$. Furthermore, for $i = 1, \dots, g$ we have
\begin{equation*}
((1+\underline{T})^{\lambda_i}-1)^{n_i} \in \mathcal{I} \quad\text{if $n_i \geq (N-1)d+1$}.
\end{equation*}
It follows that
\begin{alignat*}{2}
G(C, \xi; \underline{T})
& \equiv A \, (1+\underline{T})^{\beta} \sum_{n_1, \dots, n_g = 0}^{(N-1)d}\,\prod_{i=1}^g a_i^{n_i} \left((1 + \underline{T})^{\lambda_i}-1\right)^{n_i} \\
& \equiv A \, (1+\underline{T})^{\beta} \sum_{n_1, \dots, n_g = 0}^{(N-1)d}\,\prod_{i=1}^g \biggl[a_i^{n_i} \sum_{k_i=0}^{n_i} (-1)^{n_i - k_i} \binom{n_i}{k_i} (1 + \underline{T})^{k_i\lambda_i}\biggr] \\
& \equiv A \sum_{k_1, \dots, k_g = 0}^{(N-1)d} (1 + \underline{T})^{\beta + \underline{k} \cdot \underline{\lambda}}\,\prod_{i=1}^g \sum_{n_i=k_i}^{(N-1)d} (-1)^{n_i-k_i} a_i^{n_i}\binom{n_i}{k_i} \pmod{\mathcal{I}}.
\end{alignat*}
Applying $\Omega$ to each side we get
\begin{align*}
F(C, \xi; T) & \equiv A \, \sum_{k_1, \dots, k_g = 0}^{(N-1)d} (1 + T)^{\mathcal{N}(\beta + \underline{k} \cdot \underline{\lambda})} \prod_{i=1}^g B_{k_i, (N-1)d}(\xi(\lambda_i)) \pmod{T^N}. \qedhere
\end{align*}
\end{proof}

\begin{corollary}\label{cor:intcoeff}
The power series $F(C, \mathfrak{c}; T)$ has coefficients in $\mathbb{Z}_p$.
\end{corollary}

\begin{proof} 
Let $\xi$ be a non-trivial element of $X(\mathfrak{c})$ and write $C = C(\beta; \lambda_1, \dots, \lambda_g)$. For $i = 1, \dots, g$ we know $1 - \xi(\lambda_i)$ is divisible only by primes above $c$ and hence is invertible in $\mathbb{Z}_p[\xi]$. Therefore the polynomials $F_N(C, \xi; T)$ have coefficients in $\mathbb{Z}_p[\xi]$, and it follows from Galois theory (see also Subsection~\ref{subsec:comzetcone}) that 
\begin{equation*}
F_N(C, \mathfrak{c}; T) := \sum_{\substack{\xi \in X(\mathfrak{c}) \\ \xi \ne 1}} F_N(C, \xi; T)
\end{equation*}
has coefficients in $\mathbb{Z}_p$.
By the theorem we have $F(C, \mathfrak{c}; T) \equiv F_N(C, \mathfrak{c}; T) \pmod{T^N}$ for all $N \geq 0$,
and therefore $F(C, \mathfrak{c}; T)$ has coefficients in $\mathbb{Z}_p$. 
\end{proof}
 
Since $F_{\mathfrak{m}}(\mathfrak{a}^{-1}, \mathfrak{c}; T)$ is the sum of the $F(C_i, \mathfrak{c}; T)$'s where $\{C_1, \dots, C_m\}$ is a $\mathfrak{c}$-admissible cone decomposition of $\mathfrak{a}$ modulo $\mathfrak{m}$, it follows from the corollary that $F_{\mathfrak{m}}(\mathfrak{a}^{-1}, \mathfrak{c}; T)$ has coefficients in $\mathbb{Z}_p$. In particular, $F_{\mathfrak{m}}(\mathfrak{a}^{-1}, \mathfrak{c}; T)$ defines a $\mathbb{Z}_p$-valued measure, which we will denote $\mu_{p,\mathfrak{m}}^{\mathfrak{a},\mathfrak{c}}$. To be able to apply Theorem~\ref{th:procinterpol} to this measure we need to prove that $\mu_{p,\mathfrak{m}}^{\mathfrak{a},\mathfrak{c}}$ has support in $1 + q\mathbb{Z}_p$. We will actually prove a stronger statement that will be useful later. Let $\mathbb{Q}_\infty$ be the cyclotomic $\mathbb{Z}_p$-extension of $\mathbb{Q}$. Denote by $\mathbb{Q}_0 := \mathbb{Q},\,\mathbb{Q}_1,\,\mathbb{Q}_2,\,\ldots$ the subfields of $\mathbb{Q}_\infty/\mathbb{Q}$ with $\Gal(\mathbb{Q}_n/\mathbb{Q}) \simeq \mathbb{Z}/p^n\mathbb{Z}$. Let $E_\infty$ be the cyclotomic $\mathbb{Z}_p$-extension of $E$ and let $E_0 := E,\,E_1,\,E_2,\,\ldots$ be the subfields of $E_\infty/E$ with $\Gal(E_n/E) \simeq \mathbb{Z}/p^n\mathbb{Z}$. Define integers $m_0, m_1 \geq 0$ by $\mathbb{Q}_{m_0} = E \cap \mathbb{Q}_\infty$ and $\mathbb{Q}_{m_0+m_1} = E(\mathfrak{m}) \cap \mathbb{Q}_\infty$. By construction $E_{m_1} = E \, \mathbb{Q}_{m_0+m_1}$ is the intersection of $E(\mathfrak{m})$ and $E_\infty$. The commutative diagram 
\begin{equation}\label{eq:cd}\raisebox{5ex}{\xymatrix{
\Cl_\mathfrak{m}(E) \ar[r]^-{\mathcal{N}} \ar[d] & (1 + qp^{m_0}\mathbb{Z})/(1 + qp^{m_0+m_1}\mathbb{Z}) \ar[d] \\
\Gal(E_{m_1}/E) \ar[r]^{\res\quad\;} & \Gal(\mathbb{Q}_{m_0+m_1}/\mathbb{Q}_{m_0}) 
}}\end{equation}
comes from Class Field Theory, where the bottom map is the restriction map, the top map is induced by the map $\mathfrak{a} \mapsto \langle \mathcal{N}(\mathfrak{a})\rangle$, and the vertical maps are the respective Artin maps. Define $e \geq 1$ to be the largest integer such that $W_{p^e} \subset E(W_q)$. It is clear that we have $e = m_0 + v_p(q)$. We note in passing the lemma below, which will be useful later and which is a direct consequence of the diagram.

\begin{lemma}\label{lem:normcong}
For any fractional ideal  $\mathfrak{a}$ of $E$ coprime to $p$ we have $\langle \mathcal{N}(\mathfrak{a}) \rangle \in 1 + p^e\mathbb{Z}_p$. \qed
\end{lemma}

We now prove our result on the support of the measure $\mu_{p,\mathfrak{m}}^{\mathfrak{a},\mathfrak{c}}$.

\begin{proposition}\label{prop:propmu}
The measure $\mu_{p,\mathfrak{m}}^{\mathfrak{a},\mathfrak{c}}$ has support in $1 + p^{e+m_1} \mathbb{Z}_p$.
\end{proposition}

\begin{proof} 
Let $C$ be a cone in a $\mathfrak{c}$-admissible cone decomposition of $\mathfrak{a}$ modulo $\mathfrak{m}$ and let $\xi$ be a non-trivial element of $X(\mathfrak{c})$. Denote by $\mu_{C, \xi}$, respectively $\mu_{C, \xi}^N$ with $N \geq 1$, the measure associated with the power series $F(C, \xi; T)$, respectively $F_N(C, \xi; T)$, and write $C = C(\beta; \lambda_1, \dots, \lambda_g)$. For $\underline{k} \in \mathbb{N}^g$ the algebraic integers $\beta + \underline{k} \cdot \underline {\lambda}$, are all congruent to $1$ modulo $q$, and thus we have $\mathcal{N}(\beta + \underline{k} \cdot \underline {\lambda}) \equiv 1 \pmod{q}$. It follows that $\mathcal{N}(\beta + \underline{k} \cdot \underline {\lambda}) = \langle \mathcal{N}(\beta + \underline{k} \cdot \underline {\lambda}) \rangle \in 1 + p^{e+m_1} \mathbb{Z}_p$ by the diagram above, since $\beta + \underline{k} \cdot \underline {\lambda} \in E_\mathfrak{m}$. Thus, by Lemma~\ref{lem:dirac}, the measures $\mu_{\xi,C}^N$ have support in $1 + p^{e+m_1}\mathbb{Z}_p$. The same is true for $\mu_{C,\xi}$ since the measures $\mu_{\xi,C}^N$ converge (weakly) to $\mu_{C, \xi}$. The conclusion follows as well for $\mu_{p,\mathfrak{m}}^{\mathfrak{a},\mathfrak{c}}$, it being the sum of finitely many such measures. 
\end{proof}

\begin{remark}
Replacing $\mathfrak{f}$ by $\mathfrak{f}p^a$ for some $a \geq 1$ does not change the $\p$-adic $L$-function, as we will see from the interpolation property \eqref{eq:interpolL} it satisfies (and the unicity statement that follows from it; see Remark~\ref{rk:unicity}). In particular, by taking $a$ large enough we can force the measures $\mu_{p,\mathfrak{m}}^{\mathfrak{a},\mathfrak{c}}$ to have support in $1 + p^b\mathbb{Z}_p$ for $b \geq 1$ arbitrarily large. The proposition shows that (H1) is enough to imply that $b$ is already quite large.
\end{remark}

\subsectb{Construction of $\p$-adic $L$-functions}{Construction of p-adic L-functions}\label{subsec:construction}

We are now ready to define $\p$-adic $L$-functions. The first step is to define the $\p$-adic equivalent of twisted partial functions.

\begin{proposition}\label{prop:defzp}
For $m$ an integer and $s \in \mathbb{Z}_p$ define $\mathcal{Z}_{p,\mathfrak{m}}^{(m)}(\mathfrak{a}^{-1}, \mathfrak{c}; s)$ to be the $\p$-adic twisted partial zeta function given by 
\begin{equation}\label{eq:defzpxi}
\mathcal{Z}_{p,\mathfrak{m}}^{(m)}(\mathfrak{a}^{-1}, \mathfrak{c}; s) := \omega(\mathcal{N}(\mathfrak{a}))^m
\langle \mathcal{N}(\mathfrak{a}) \rangle^s \int \phi_{-s}(x) \, d\mu_{p, \mathfrak{m}}^{\mathfrak{a}, \mathfrak{c}}.
\end{equation}
Then $\mathcal{Z}_{p,\mathfrak{m}}^{(m)}(\mathfrak{a}^{-1}, \mathfrak{c}; s)$ is a continuous function on $\mathbb{Z}_p$, and
\begin{equation*}
\mathcal{Z}_{p,\mathfrak{m}}^{(m)}(\mathfrak{a}^{-1}, \mathfrak{c}; -k) = Z_{\mathfrak{m}}(\mathfrak{a}^{-1}, \mathfrak{c}; -k)
\end{equation*}
for all $k \in \mathbb{N}$ such that $k+m \equiv 0 \pmod{\varphi(q)}$.
\end{proposition}

\begin{remark}
Using Proposition~\ref{prop:propmu} we could restrict the domain of integration in \eqref{eq:defzpxi} to $1+p^{e+m_1}\mathbb{Z}_p$ and then replace $\phi_{-s}(x)$ by $x^{-s}$. This would give a somewhat neater formula, and we will actually use this expression several times in the next subsection. However, from a computational point view it is better to express things as in \eqref{eq:defzpxi}, since that is how the computation will actually be done.
\end{remark}

\begin{proof}
This is a direct application of Theorem~\ref{th:procinterpol} (changing $s$ to $-s$) with $a_n = Z_{\mathfrak{m}}(\mathfrak{a}^{-1}, \mathfrak{c}; -n)$ and $F = F_{\mathfrak{m}}(\mathfrak{a}^{-1}, \mathfrak{c}; T)$, using Corollary~\ref{cor:interpolF} and Proposition~\ref{prop:propmu} and the fact that for $k \in \mathbb{N}$ with $k+m \equiv 0 \tmod{\varphi(q)}$ we have
\begin{equation*}
\omega(\mathcal{N}(\mathfrak{a}))^m \langle \mathcal{N}(\mathfrak{a}) \rangle^{-k} =  \mathcal{N}(\mathfrak{a})^{-k}.  \qedhere
\end{equation*}
\end{proof}

Let $\chi$ be a complex character on $\Cl_\mathfrak{m}(E)$. Recall that we have embedded $\bar{\mathbb{Q}}$ into $\mathbb{C}_p$ and thus we can view $\chi$ also as a $\p$-adic character. Define the character $\kappa$ of $\Cl_\mathfrak{m}(E)$ by the composition 
\begin{equation*}\xymatrix{
\kappa : \Cl_\mathfrak{m}(E) \ar@{->>}[r] & \Cl_q(E) \ar[r]^-\simeq & \Gal(E(q)/E) \ar[r] & (\mathbb{Z}/q\mathbb{Z})^\times \ar[r]^-\simeq & T_p. 
}\end{equation*}
The first map is the natural surjection coming from the fact that $q$ divides $\mathfrak{m}$ by (H1). The second sends a class $\mathcal{C}$ to its Artin symbol $\sigma_\mathcal{C}$, where $E(q)$ is the ray-class field of $E$ modulo~$q$. The next comes from the fact that $E(q)$ contains the $q$-th root of unity, and thus we can associate with any $\sigma \in \Gal(E(q)/E)$ a class $\bar{a}$ in $(\mathbb{Z}/q\mathbb{Z})^\times$ such that $\sigma(\zeta) = \zeta^a$ for all $\zeta \in W_q$. The last map sends $\bar{a}$ to $\omega(a) \in T_p$. For a fractional ideal $\mathfrak{a}$ of $E$, relatively prime to $\mathfrak{m}$, it follows from the definition of the Artin map that $\kappa(\mathfrak{a}) = \omega(\mathcal{N}(\mathfrak{a}))$. As a consequence of the previous result and of Proposition~\ref{prop:lfunc}, we recover the construction of $\p$-adic $L$-functions. 

\begin{theorem}[Deligne-Ribet, Barsky, Cassou-Noguès]
Let $m$ be an integer, and if $\chi \ne \kappa^{1-m}$ assume further that $\mathfrak{c}$ is such that $\chi(\mathfrak{c}) \ne \kappa^{1-m}(\mathfrak{c})$. Define a $\p$-adic $L$-function by
\begin{equation}\label{def:padicL}
L^{(m)}_{p, \mathfrak{m}}(\chi; s) := 
\left(\frac{\chi(\mathfrak{c})}{\omega(c)^{m-1} \langle c \rangle^{s-1}} - 1 \right)^{-1}
\,\sum_{i=1}^{h_\mathfrak{m}(E)} \chi(\mathfrak{a}_i^{-1}) \mathcal{Z}^{(m)}_{p,\mathfrak{m}}(\mathfrak{a}_i^{-1}, \mathfrak{c}; s)
\end{equation}
where the sum is over integral ideals $\mathfrak{a}_i$, relatively prime to $\mathfrak{m}$ and $\mathfrak{c}$, representing all the classes of $\Cl_\mathfrak{m}(E)$. Then $L^{(m)}_{p, \mathfrak{m}}(\chi; s)$ is a continuous function on $\mathbb{Z}_p$  \uppar{respectively $\mathbb{Z}_p \smallsetminus \{1\}$ if $\chi = \kappa^{1-m}$} and
\begin{equation}\label{eq:interpolL}
\vdepth{8pt}
L^{(m)}_{p,\mathfrak{m}}(\chi; -k) = L_\mathfrak{m}(\chi; -k)
\end{equation}
for all $k \in \mathbb{N}$ such that $k+m \equiv 0 \tmod{\varphi(q)}$.
\end{theorem}

\begin{remark}\label{rk:constc}
Assume $\chi\kappa^{m-1}$ is not the trivial character and let $N$ be the Galois closure of $E(\mathfrak{m})/\mathbb{Q}$. Let $\tsigma \in \Gal(E(\mathfrak{m})/E)$ be such that $\chi\kappa^{m-1}(\tsigma) \ne 1$.\footnote{We see $\chi\kappa^{m-1}$ as a character on $\Gal(E(\mathfrak{m})/E)$ via the Artin map.} Lift $\tsigma$ to an arbitrary element $\sigma$ of $\Gal(N/\mathbb{Q})$. By the theorem of \u{C}ebotarev \cite[Chap.~VIII, Th.~13.4]{neukirch} there exists a positive proportion of prime ideals of $N$ whose Frobenius in $N/\mathbb{Q}$ is $\sigma$. Let $\mathfrak{C}$ be one of these prime ideals with $\mathfrak{C}$ coprime to $\mathfrak{f}\mathbb{Z}_N$. Then $\mathfrak{c} := \mathfrak{C} \cap \mathbb{Z}_E$ is a prime ideal of $E$, coprime to $\mathfrak{f}$, of residual degree $1$, such that $\chi\kappa^{m-1}(\mathfrak{c}) \ne 1$.
\end{remark}
 
\begin{proof}
By the previous result the sum in the RHS of \eqref{def:padicL} is a continuous function on $\mathbb{Z}_p$. We now look at the factor before the sum. It is continuous at $s \in \mathbb{Z}_p$ unless $\chi(\mathfrak{c}) = \omega(c)^{m-1} \langle c\rangle^{s-1}$. Since $\langle c\rangle$ has infinite order, this can happen only for $s = 1$. Thus, $L^{(m)}_{p, \mathfrak{m}}(\chi; s)$ is continuous on $\mathbb{Z}_p \smallsetminus \{1\}$, and also at $s = 1$ if $\chi(\mathfrak{c}) \ne \omega(c)^{1-m}$. We get the first result by noting that $ \omega(c) = \kappa(\mathfrak{c})$. For the second, let $k \in \mathbb{N}$ be such that $k + m \equiv 0 \tmod{q}$. The interpolation property follows from Proposition~\ref{prop:defzp}, Proposition~\ref{prop:lfunc}, and the fact that 
\begin{equation*}
\frac{\chi(\mathfrak{c})}{\omega(c)^{m-1} \langle c \rangle^{k-1}} =  \chi(\mathfrak{c}) c^{1-k}.  \qedhere
\end{equation*}
\end{proof}

\begin{remark}\label{rk:chiczero}
Assume $\chi\kappa^{m-1}$ is non-trivial but $\chi\kappa^{m-1}(\mathfrak{c}) = 1$. Then \eqref{def:padicL} still defines a continuous function on $\mathbb{Z}_p \smallsetminus \{1\}$. Moreover, that function is equal to $L^{(m)}_{p, \mathfrak{m}}(\chi; s)$ if $s \ne 1$, and therefore, by continuity, it converges to $L^{(m)}_{p, \mathfrak{m}}(\chi; 1)$ as $s$ tends to $1$.
\end{remark}

\begin{remark}\label{rk:unicity}
Assume $p$ is odd. Then the set $\{-k$ with $k \in \mathbb{N}$ and $k + m \equiv 0 \tmod{p-1}\}$ is dense in $\mathbb{Z}_p$, and there is a unique continuous $\p$-adic function satisfying \eqref{eq:interpolL}. This proves that $L^{(m)}_{p,\mathfrak{m}}(\chi; s)$ does not depend on the choice of $\mathfrak{c}$. For $p = 2$ the closure of $\{-k$ with $k \in \mathbb{N}$ and $k + m \equiv 0 \tmod{2}\}$ is either $2\mathbb{Z}_2$ or $1 + 2\mathbb{Z}_2$. Thus \eqref{eq:interpolL} is not enough to prove unicity. In that case we can use Proposition~\ref{prop:Lpanal} below\footnote{The proposition only applies to $m \equiv 1 \tmod{2}$, but it is straightforward to generalize.} to conclude that $L^{(m)}_{2,\mathfrak{m}}(\chi; s)$ does not depend on the choice of $\mathfrak{c}$.
\end{remark}

There are actually $\varphi(q)$ twisted partial zeta functions or $L$-functions defined by these two results, depending on the choice of the congruence class of $m$ modulo $\varphi(q)$. However, when $\chi$ is the trivial character we see that only for $m \equiv 1 \tmod{\varphi(q)}$ does the corresponding $\p$-adic zeta function possibly have a pole at $s = 1$. (See \cite{colmez:residue} for the computation of the residue; the fact that it is non-zero is equivalent to the Leopoldt conjecture.) Therefore, \emph{the} $\p$-adic $L$-function, denoted simply $L_{p,\mathfrak{m}}(\chi; s)$, is defined to be the function corresponding to the choice $m \equiv 1 \tmod{\varphi(q)}$. From now on we will focus uniquely on that case, dropping the exponent in the notation and writing $\mathcal{Z}_{p,\mathfrak{m}}(\mathfrak{a}^{-1}, \mathfrak{c}; s)$ instead of $\mathcal{Z}_{p,\mathfrak{m}}^{(1)}(\mathfrak{a}^{-1}, \mathfrak{c}; s)$, and so on. We will see below (see Proposition~\ref{prop:changem}) that the different $L$-functions for various $m$ can be recovered from $L_{p,\mathfrak{m}}(\chi; s)$ by twisting the character $\chi$ by some appropriate power of $\kappa$.

\subsectb{Some properties of $\p$-adic $L$-functions}{Some properties of p-adic L-functions} 

In this subsection we prove some well-known results about $\p$-adic $L$-functions that will be useful later. The first is a direct consequence of Proposition~\ref{prop:defzp} (and the remark that follows it).
\begin{proposition}
The Mahler expansion of the $\p$-adic twisted partial zeta function is
\begin{equation*}
\mathcal{Z}_{p,\mathfrak{m}}(\mathfrak{a}^{-1}, \mathfrak{c}; s) = \omega(\mathcal{N}(\mathfrak{a})) \sum_{n \geq 0} \int_{1+p^{e+m_1}\mathbb{Z}_p} \left(\big\langle \mathcal{N}(\mathfrak{a})\big\rangle x^{-1} - 1\right)^n \, d\mu_{p, \mathfrak{m}}^{\mathfrak{a}, \mathfrak{c}} \, \binom{s}{n}. \tag*{\hspace{-1em}\qed}
\end{equation*}
\end{proposition}

From this, we deduce the analyticity of $\p$-adic $L$-functions.

\begin{proposition}\label{prop:Lpanal}
Let $\mathcal{B}_e$ be the open ball in $\mathbb{C}_p$ of center $0$ and radius $p^{e-1/(p-1)}$. Then the $\p$-adic $L$-function $L_{p,\mathfrak{m}}(\chi; s)$ can be extended to an analytic function on $\mathcal{B}_e$, if $\chi$ is non-trivial, and to  a meromorphic function on $\mathcal{B}_e$ with a pole of order at most $1$ at $s=1$, if $\chi$ is trivial.
\end{proposition}

\begin{proof}
We first prove that the $\p$-adic twisted partial zeta functions $\mathcal{Z}_{p,\mathfrak{m}}(\mathfrak{a}^{-1}, \mathfrak{c}; s)$ can be extended to analytic functions of radius\footnote{We say that an analytic function has radius $r$ if it converges on the \emph{open} ball in $\mathbb{C}_p$ of center $0$ and radius~$r$.} $p^{e-1/(p-1)}$. For $x \in 1+p^{e+m_1}\mathbb{Z}_p$ it follows from Lemma~\ref{lem:normcong} that $\left|\left(\big\langle \mathcal{N}(\mathfrak{a})x^{-1} \big\rangle - 1\right)^n \right|_p \leq p^{-en}$, and since the measure $\mu_{p, \mathfrak{m}}^{\mathfrak{a}, \mathfrak{c}}$ is of norm $\leq 1$ and $F_\mathfrak{m}(\mathfrak{a}^{-1}, \mathfrak{c}; T)$ has coefficients in $\mathbb{Z}_p$ we conclude by the previous proposition that the $n$-th Malher coefficient of $\mathcal{Z}_{p,\mathfrak{m}}(\mathfrak{a}^{-1}, \mathfrak{c}; s)$ has $\p$-adic absolute value $\leq p^{-en}$. The result then follows from Corollaire 2(c) of \cite[p.~162]{amice}.  (See also Theorem~\ref{th:amice}.)

Denoting by $g(s)$ the inverse of the factor before the sum in the RHS of \eqref{def:padicL} we have
\begin{equation*}
g(s) := \chi(\mathfrak{c}) \langle c \rangle^{1-s} - 1 = \chi(\mathfrak{c}) \, \exp_p((1-s) \log_p\langle c \rangle) - 1 
\end{equation*}
where $\exp_p$ and $\log_p$ are respectively the $\p$-adic exponential and logarithm functions (see \cite[\S5.4]{robert}). The function $g(s)$ is analytic on $\mathcal{B}_e$ since $|\log_p  \langle c \rangle| \leq q^{-e}$, using again Lemma~\ref{lem:normcong} and the fact that the $\p$-adic exponential function has radius $p^{-1/(p-1)}$. Furthermore, from the properties of the $\p$-adic exponential and logarithm functions we see that $g$ has a simple zero at $s = 1$ if $\chi(\mathfrak{c}) = 1$, and does not vanish otherwise. Now if $\chi$ is non-trivial we can proceed as in Remark~\ref{rk:constc} and choose $\mathfrak{c}$ such that $\chi(\mathfrak{c}) \ne 1$; this proves the result for the first case. Otherwise we write $g(s) = (s-1) h(s)$ where $h(s)$ is an analytic function non-vanishing on $\mathcal{B}_e$ and the second case follows.
\end{proof}

The next result establishes that the $\p$-adic $L$-function $L^{(m)}_{p,\mathfrak{m}}(\chi; s)$ for any $m$ in $\mathbb{Z}$ can be recovered from \emph{the} $\p$-adic $L$-function (corresponding to $m = 1$).

\begin{proposition}\label{prop:changem}
For any integer $m$ and any $s \in \mathbb{Z}_p$, assuming $s \ne 1$ if $\chi = \kappa^{1-m}$, we have
\begin{equation*} 
L^{(m)}_{p,\mathfrak{m}}(\chi; s) = L_{p,\mathfrak{m}}(\chi\kappa^{1-m}; s).
\end{equation*}
\end{proposition}

\begin{proof} 
Let $\mathfrak{a}$ be an integral ideal coprime to $p$. Then $\mathcal{Z}^{(m)}_{p,\mathfrak{m}}(\mathfrak{a}^{-1}, \mathfrak{c}; s) = \kappa(\mathfrak{a})^{m-1} \mathcal{Z}_{p,\mathfrak{m}}(\mathfrak{a}^{-1}, \mathfrak{c}; s)$, since $\kappa(\mathfrak{a}) = \omega(\mathcal{N}(\mathfrak{a}))$, and substitution in \eqref{def:padicL} yields the result.  
\end{proof}

A key property of $\p$-adic $L$-functions is that they are Iwasawa analytic functions (see \cite{ribet}); we will show this in the proof of the next theorem. Let $E_\chi$ be the subextension of $E(\mathfrak{m})/E$ fixed by the kernel of $\chi$. After Greenberg \cite{greenberg} we say that $\chi$ is of \emph{type $W$} if $E_\chi \subset E_\infty$.\footnote{See the discussion after Corollary~\ref{cor:intcoeff} for the notation used in the theorem and its proof.}

\begin{theorem}[Deligne-Ribet]\label{thm:iwasawaseries}
Fix a topological generator $u$ of $1 + p^e\mathbb{Z}_p$. Then there exists a unique power series $\mathfrak{I}_{p,\mathfrak{m}}(\chi; X)$ in $\mathbb{Q}_p(\chi)[[X]]$\,---\,or $X^{-1}\, \mathbb{Q}_p(\chi)[[X]]$ if $\chi$ is trivial\,---\,called the \emph{Iwasawa power series of $\chi$} \uppar{with respect to $u$} such that, for all $s \in \mathbb{Z}_p$ \uppar{excepting $s = 0$ if $\chi$ is trivial} we have
\begin{equation}\label{eq:iwasawaseries}
L_{p,\mathfrak{m}}(\chi; 1-s) = \mathfrak{I}_{p,\mathfrak{m}}(\chi; u^s-1).
\end{equation}
Moreover,
if $\chi$ is trivial or not of type $W$ then $\mathfrak{I}_{p,\mathfrak{m}}(\chi; X)$ has coefficients in $\mathbb{Z}_p[\chi]$. Otherwise there exists a non-trivial root of unity $\xi$ of order dividing $p^{m_1}\!$ such that $(\xi(1+X)-1) \mathfrak{I}_{p,\mathfrak{m}}(\chi; X)$ has coefficients in $\mathbb{Z}_p[\chi]$.
\end{theorem}

\begin{proof}
Unicity is clear by \eqref{eq:iwasawaseries} since the set $\{u^s - 1$ with $s \in \mathbb{Z}_p\} = p^e \mathbb{Z}_p$ admits $0$ as a limit point. In particular, if the Iwasawa power series exits it does not depend on the choice of $\mathfrak{c}$. We will use this fact below by choosing prime ideals $\mathfrak{c}$ satisfying additional properties. We now prove existence. For $x \in \mathbb{Z}_p^\times$ with $\langle x \rangle \in 1 + p^e \mathbb{Z}_p$ define 
\begin{equation*}
\mathcal{L}_u(x) := \frac{\log_p \langle x \rangle}{\log_p u} \in \mathbb{Z}_p,
\end{equation*}
so that
\begin{equation*}
\langle x\rangle^s = \big(u^{\mathcal{L}_u(x)}\big)^{s} = \sum_{\ell \geq 0} (u^s-1)^\ell \binom{\mathcal{L}_u(x)}{\ell}.
\end{equation*}
In particular, we can use this equation with $x := \mathcal{N}(\mathfrak{a})$, for some ideal $\mathfrak{a}$ coprime to $p$, by Lemma~\ref{lem:normcong}. We define three power series, with coefficients in $\mathbb{Z}_p$, $\mathbb{Z}_p$, and $\mathbb{Z}_p[\chi]$ respectively, as follows.
\begin{align*}
N(\mathfrak{a}; X) & := \mathcal{N}(\mathfrak{a}) \sum_{\ell \geq 0} \binom{-\mathcal{L}_u(\mathcal{N}(\mathfrak{a}))}{\ell} \, X^\ell  = \mathcal{N}(\mathfrak{a}) (1+X)^{-\mathcal{L}_u(\mathcal{N}(\mathfrak{a}))}, \\
A(\mathfrak{a}, \mathfrak{c}; X) & := \sum_{\ell \geq 0} \int_{1+p^{e+m_1} \mathbb{Z}_p} x^{-1} \binom{\mathcal{L}_u(x)}{\ell} \, d\mu_{p, \mathfrak{m}}^{\mathfrak{a}, \mathfrak{c}} \,\, X^\ell = \int_{1+p^{e+m_1} \mathbb{Z}_p} x^{-1} (1+X)^{\mathcal{L}_u(x)} \, d\mu_{p,\mathfrak{m}}^{\mathfrak{a}, \mathfrak{c}}, \\
C(\mathfrak{c}, \chi; X) & := \chi(\mathfrak{c}) \sum_{\ell \geq 0} \binom{\mathcal{L}_u(c)}{\ell} X^\ell - 1 = \chi(\mathfrak{c}) (1+X)^{\mathcal{L}_u(c)} - 1.
\end{align*}
Then for all $s \in \mathbb{Z}_p$, using \eqref{eq:defzpxi} for the first equality, we have
\begin{align*}
N(\mathfrak{a}; u^s - 1) A(\mathfrak{a}, \mathfrak{c}; u^s - 1) & = \mathcal{Z}_{p,\mathfrak{m}}(\mathfrak{a}^{-1}, \mathfrak{c}; 1-s), \\ 
C(\mathfrak{c}, \chi; u^s - 1) & = \chi(\mathfrak{c}) \langle c \rangle^s - 1.
\end{align*}
We define
\begin{equation}\label{eq:compiwaseries}
\mathfrak{I}_{p,\mathfrak{m}}(\chi; X) := C(\mathfrak{c}, \chi; X)^{-1} \sum\limits_{i=1}^{h_\mathfrak{m}(E)} \chi(\mathfrak{a}_i^{-1}) N(\mathfrak{a}_i; X) A(\mathfrak{a}_i, \mathfrak{c}; X) 
\end{equation}
where the sum is over integral ideals $\mathfrak{a}_i$, relatively prime to $\mathfrak{c}$, representing all the classes of $\Cl_\mathfrak{m}(E)$. For all $s \in \mathbb{Z}_p$ such that $C(\mathfrak{c}, \chi; u^s - 1) \ne 0$ it follows from \eqref{def:padicL} and the equalities above that \eqref{eq:iwasawaseries} holds.

We now consider several cases, not necessarily disjoint. Assume first that the order of $\chi$ is not a power of $p$. We reason as in Remark~\ref{rk:constc} with $\tsigma \in \Gal(E(\mathfrak{m})/E)$ such that the order of $\chi(\tsigma)$ is not a power of $p$. Then $\chi(\mathfrak{c}) - 1$, the constant coefficient of $C(\mathfrak{c}, \chi; X)$, is a $\p$-adic unit. Therefore $C(\mathfrak{c}, \chi; X)$ does not vanish on $p^e\mathbb{Z}_p$ and is invertible in $\mathbb{Z}_p[\chi][[X]]$. This proves that $\mathfrak{I}_{p,\mathfrak{m}}(\chi; X) \in \mathbb{Z}_p[\chi][[X]]$ and satisfies \eqref{eq:iwasawaseries} for all $s \in \mathbb{Z}_p$.

Assume now that $\chi$ is such that $E_\chi \cap E_\infty = E$. Let $N$ be the Galois closure over $\mathbb{Q}$ of the compositum of $E_\chi$ and $E_1$. Then there exists $\sigma \in \Gal(N/\mathbb{Q})$ such that $\sigma_{|E_\chi}$ is trivial but $\sigma_{|E_1}$ is non-trivial. Let $\mathfrak{C}$ be a prime ideal of $N$, coprime to $\mathfrak{f}\mathbb{Z}_N$, whose Frobenius is equal to $\sigma$, and let $\mathfrak{c} := \mathfrak{C} \cap \mathbb{Z}_E$. Then $\mathfrak{c}$ has residual degree $1$, and $\chi(\mathfrak{c}) = 1$ so the constant term of $C(\mathfrak{c}, \chi; X)$ is zero. Also, by construction we have $\mathcal{L}_u(c) \in \mathbb{Z}_p^\times$,
and since this is the coefficient of $X$ in $C(\mathfrak{c}, \chi; X)$ it follows that $C(\mathfrak{c}, \chi; X) = X U(X)$, where $U(X) \in \mathbb{Z}_p[\chi][[X]]$ is an invertible power series. Therefore $C(\mathfrak{c}, \chi; X)$ vanishes only at $0$, that is, for $s = 0$, and $\mathfrak{I}_{p,\mathfrak{m}}(\chi; X) \in X^{-1} \mathbb{Z}_p[\chi][[X]]$. This proves the result when $\chi$ is trivial. When $\chi$ is non-trivial, the limit of $\mathfrak{I}_{p,\mathfrak{m}}(\chi; X)$ when $X \to 0$ exists and is finite.\footnote{It is equal to $L_{p,\mathfrak{m}}(\chi; 1)$.} Therefore the coefficient of $X^{-1}$ in $\mathfrak{I}_{p,\mathfrak{m}}(\chi; X)$ is zero. This proves the result in the non-trivial case.

Before looking at the other cases we need additional notation and results. Recall that $m_0$ and $m_1$ are such that $\mathbb{Q}_{m_0} = E \cap \mathbb{Q}_\infty$ and $\mathbb{Q}_{m_0+m_1} = E(\mathfrak{m}) \cap \mathbb{Q}_\infty$. Let $\xi_0$ be a root of unity of order $p^{m_1}$. Define a function on $1 + qp^{m_0}\mathbb{Z}_p$ by
\begin{equation*}
a \mapsto \xi_0^{(\log_p a)/(qp^{m_0})}.
\end{equation*}
This map is a group homomorphism with kernel $1 + qp^{m_0+m_1}\mathbb{Z}_p$. Composing with the function on the top in diagram \eqref{eq:cd} we get a character 
\begin{equation*}
\mathfrak{a} \mapsto \xi_0^{\log_p \langle \mathcal{N}(\mathfrak{a})\rangle/(qp^{m_0})}
\end{equation*}
on the ray-class group $\Cl_\mathfrak{m}(E)$. It is easy to see from its construction that this character generates the subgroup of characters of $\Cl_\mathfrak{m}(E)$ of type $W$. Since $e = m_0 +v_p(q)$, we have $\log_p u/(qp^{m_0}) \in \mathbb{Z}_p^\times$ and without loss of generality we can replace $\xi_0$ by $\xi_0^{\log_p u/(qp^{m_0})}$ to get the character $\rho : \mathfrak{a} \mapsto \xi_0^{\mathcal{L}_u(\mathcal{N}(\mathfrak{a}))}\!$, which still generates the group of characters of $\Cl_\mathfrak{m}(E)$ of type~$W$. For $v \in \mathbb{Z}$ we compute
\begin{equation*}
A(\mathfrak{a}, \mathfrak{c}; \xi_0^v(1+X)-1) = \int_{1+p^{e+m_1} \mathbb{Z}_p} x^{-1} (\xi_0^v(1+X))^{\mathcal{L}_u(x)} \, d\mu_{p,\mathfrak{m}}^{\mathfrak{a}, \mathfrak{c}} = A(\mathfrak{a}, \mathfrak{c}; X). 
\end{equation*}
We also have $N(\mathfrak{a}; \xi_0^v(1+X) - 1) = \rho^{-v}(\mathfrak{a}) N(\mathfrak{a}; X)$ and, for any character $\psi$ of $\Cl_\mathfrak{m}(E)$, $C(\mathfrak{c}, \psi; \xi_0^v(1+X)-1) = C(\mathfrak{c}, \psi\rho^v; X)$. We conclude that
\begin{equation}\label{eq:chgvariwa}
\mathfrak{I}_{p,\mathfrak{m}}(\psi\rho^v; X) = \mathfrak{I}_{p,\mathfrak{m}}(\psi; \xi_0^v(1+X) - 1).
\end{equation}

Assume now that $\chi$ is of type $W$. Then $\chi = \rho^v$ for some $v \in \mathbb{Z}$. Using \eqref{eq:chgvariwa} with $\psi$ the trivial character, we find that $(\xi_0^v(1+X)-1) \mathfrak{I}_{p,\mathfrak{m}}(\chi; X) \in \mathbb{Z}_p[\chi][[X]]$, which proves the result in this case, taking $\xi := \xi_0^v$.

Finally we consider the case $E_\chi \not\subset E_\infty$, in which $\chi$ is not of type $W$ and, in particular, $E_\chi \cap E_\infty \ne E$. In this case we can write $\chi = \psi\rho^v$ for some non-trivial character $\psi$ of $\Cl_\mathfrak{m}(E)$ satisfying $E_\psi \cap E_\infty = E$ and some $v \in \mathbb{Z}$. Since the Iwasawa power series for $\psi$ is in $\mathbb{Z}_p[\chi][[X]]$ by the argument above, it follows from \eqref{eq:chgvariwa} that the same is true for $\mathfrak{I}_{p,\mathfrak{m}}(\chi; X)$.
\end{proof}

As a first application we use this result to bound the size of the values of $\p$-adic $L$-functions.

\begin{corollary}\label{cor:boundc}
Let $s \in \mathbb{Z}_p$, with $s \ne 1$ if $\chi$ is trivial. Then
\begin{equation*}
|L_{p, \mathfrak{m}}(\chi; s)|_p \leq
\begin{cases}
1 & \text{if $\chi$ is not of type $W$}, \\
p^{-1/(p-1)} &  \text{if $\chi$ is of type $W$ and non-trivial}, \\
p^e |1-s|_p^{-1} &  \text{if $\chi$ is trivial}.
\end{cases}
\end{equation*}
\end{corollary}

\begin{proof}
The result is clear if $\chi$ is not of type $W$.\footnote{In this case we also get the well-known statement: $L_{p, \mathfrak{m}}(\chi; s) \equiv L_{p, \mathfrak{m}}(\chi; 0) \pmod{p^e}$.} Assume $\chi$ is of type $W$ and non-trivial. In the notation of the theorem we have $v_p(\xi - 1) \leq 1/(p-1)$ and $v_p(u^{1-s} - 1) \geq v_p(q)$. From the identity $ab - 1 = (a-1)b + b - 1$ we get $v_p(\xi u^{1-s} - 1) \leq 1/(p-1)$ and the result follows. For trivial $\chi$ we have $v_p(u^{1-s}-1) = v_p(1-s) + e$, which proves the result.
\end{proof}

\section{Computational methods}

In this last section we show how to use the results of the previous sections to compute values and representations of $\p$-adic $L$-functions explicitly. Note that all computations will involve only $\p$-adic integers\,---\,approximated by integers as we explain below\,---\,and that we will need to deal with $\p$-adic rational numbers only in the last subsection (and we will do it somewhat indirectly).

In discussing the computation of $\p$-adic approximations we will follow certain terminological conventions. Let $M \geq 1$ be an integer. For $a \in \mathbb{Z}_p$ we denote by $a \smod p^M$ the unique integer $\ta$ with $0 \leq \ta < p^M$ such that $|a-\ta|_p \leq p^{-M}$. By \emphq{computing $a$ \toprec{p^M}} we mean computing $a \smod p^M$. Let $L$ be a finite dimensional $\mathbb{Z}_p$-lattice with $\ldss{v_1}{v_n}$ a (fixed) basis of $L$. For $\alpha \in L$, by \emphq{computing $\alpha$ \toprec{p^M\!} (with respect to the basis $\ldss{v_1}{v_n}$)} we mean computing the $\p$-adic numbers $\ldss{a_1}{a_n}$ \toprec{p^M}, where $\alpha = a_1 v_1 + \cdots + a_n v_n$. Note that in what follows the basis $\ldss{v_1}{v_n}$ is usually not stated explicitly but it should be clear from the context what it is. Let $N \geq 1$ be an integer and let $F(T) \in \mathbb{Z}_p[[T]]$ be a power series. By \emphq{computing $F$ \toprec{(p^M, T^N)}} we mean computing the first $N$ coefficients of $F$ \toprec{p^M}. Let $\mu \in \mathcal{M}(\mathbb{Z}_p, \mathbb{Z}_p)$ be a measure with values in $\mathbb{Z}_p$. Its associated power series $F_\mu(T)$ must therefore lie in $\mathbb{Z}_p[[T]]$. By \emphq{computing $\mu$ \toprec{(p^M, T^N)}} we mean computing the power series $F_\mu$ \toprec{(p^M, T^N)}. Finally, let $f$ be a continuous function in $\mathcal{C}(\mathbb{Z}_p, \mathbb{Z}_p)$. Then its Mahler coefficients $(f_n)_{n \geq 0}$ are all $\p$-adic integers and tend $\p$-adically to zero. For $M \geq 1$ we denote by $N_f(M)$ the smallest integer $N \geq 0$ such that $|f_n|_p \leq p^{-M}$ for all $n \geq N$. By \emphq{computing $f$ \toprec{p^M}} we mean finding an integer $N \geq N_f(M)$ and computing the coefficients $\ldss{f_0}{f_{N-1}}$ \toprec{p^M}.

In the complexity estimates below it is assumed that fast multiplication algorithms are used. Therefore, for example, it takes $\tO(M \log p)$ bit operations to multiply two rational $\p$-adic integers \toprec{p^M}, and it takes $\tO(N M \log p)$ bit operations to multiply two power series in $\mathbb{Z}_p[[T]]$ \toprec{(p^M, T^N)}.
Here, to simplify the complexity expressions, we have used \emph{$\tO$-notation}: $g \in \tO(f)$ if there exists $c > 0$ such that $g \in O(f (\log f)^c)$.

Finally, we will assume that the necessary data to work in the field $E$ have been computed. In particular, we assume that an integral basis, say $(\theta_1, \dots, \theta_d)$, is known. We will express the elements of $E$ with respect to this basis. Also, we assume that the class group, the group of units, and the ray-class group modulo $\mathfrak{m}$ are known. Algorithms to perform these tasks can be found in \cite{cohen:book1} and \cite{cohen:book2}; see also \cite{belabas:topics}.

\subsection{Computations with continuous functions}\label{subsec:compcontfunc}

Let $f \in \mathcal{C}(\mathbb{Z}_p, \mathbb{Z}_p)$. For $N \geq 1$, we compute the first $N$ Mahler coefficients of $f$ with the following algorithm.

\begin{algo}[Computation of Mahler coefficients]\label{algo:mahlercoeff}\ 
\begin{description}\itemsep2pt
\item[\textbf{Input:}] $f \in \mathcal{C}(\mathbb{Z}_p, \mathbb{Z}_p)$.
\item[\textbf{Output:}] The first $N$ Mahler coefficients of $f$ \toprec{p^M}.
\item[\textbf{1.}] For $n = 0$ to $N-1$, do $\tf_n \gets f(n) \smod p^M$.
\item[\textbf{2.}] For $j = 1$ to $N-1$, for $n = N-1$ to $j$ (decreasing), do $\tf_n \gets \tf_n - \tf_{n-1} \smod p^M$.
\item[\textbf{3.}] Return $\ldss{\tf_0}{\tf_{N-1}}$.
\end{description}
\end{algo}

\begin{lemma}\label{lem:costphiscoef}
Assume for $x \in \mathbb{Z}_p$ that it takes $O(C)$ bit operations to compute $f(x)$ \toprec{p^M}. Then Algorithm~\ref{algo:mahlercoeff} computes the first $N$ Mahler coefficients of $f$ \toprec{p^M} in $O(NC+N^2 M \log p)$ bit operations. In particular, for $s \in \mathbb{Z}_p$ it takes $\tO(N M^2 \log^2 p + N^2 M \log p)$ bit operations to compute the first $N$ Mahler coefficients of $\phi_s$ \toprec{p^M}.
\end{lemma}

\begin{proof}
Let $\tf_n^{(j)}$ denote the value of $\tf_n$ after $j$ iterations of the main loop in Step 2. We claim for $0 \leq j \leq N-1$ that
\begin{equation*}
\tf_n^{(j)} = 
\begin{cases}
(\nabla^n f)(0) \smod p^M & \text{ for } 0 \leq n \leq j, \\
(\nabla^j f)(n-j) \smod p^M & \text{ for } j \leq n \leq N-1,
\end{cases}
\end{equation*}
where $\nabla$ is the finite-difference operator defined by $(\nabla f)(s) := f(s+1) - f(s)$. The claim follows for $j=0$ by the initialization in Step 1 since $\nabla^0$ is the identity. Assume now that the claim holds for some $j$. If $0 \leq n \leq j$ then $\tf_n^{(j+1)} = \tf_n^{(j)}$ and the result is proved. If $n \geq j+1$ then
\begin{align*}
\tf_n^{(j+1)} & = \tf_n^{(j)} - \tf_{n-1}^{(j)} \smod p^M = (\nabla^j f)(n-j) - (\nabla^j f)(n-j-1) \smod p^M \\
& =  (\nabla (\nabla^j f))(n-j-1) \smod p^M = (\nabla^{j+1} f)(n-(j+1)) \smod p^M
\end{align*}
and the result follows by induction. In particular, at the end of the algorithm we have $\tf_n = (\nabla^n f)(0) \smod p^M = f_n \smod p^M$, where $(f_n)_{n \geq 0}$ are the Mahler coefficients of $f$ (see \cite[\S2.4]{robert}). This proves that the algorithm returns the correct result. We now estimate its complexity. The initial step takes $O(NC)$ bit operations by definition and the second step takes $O(N^2 M \log p)$.\footnote{The cost of computing 
the remainder modulo $p^M$ is also $O(M \log p)$, since $-p^M \leq \tf_n - \tf_{n-1} \leq p^M$.} This proves the first complexity statement. 

Now we turn to the computation of $\phi_s(x) \smod p^M$. We assume $s$ is given by its approximation $s \smod p^M$, which we will still denote $s$ by abuse. We can also replace $x$ by $x \smod p^M$  without loss of generality. If $p$ divides $x$ then $\phi_s(x) = 0$. We now suppose that $x \in \mathbb{Z}_p^\times$. For $p$ odd, assume we have computed and stored the values $\omega(a) \smod p^M$, for $a = 1, \dots, p-1$.\footnote{These can be easily computed using Hensel's Lemma.} Then we can compute $\omega(x) \smod p^M$ in $\tO(M \log p)$ bit operations, since $\omega(x) = \omega(a)$, where $a := x \smod p$. The computation of $\omega(x)$ for $p = 2$ is trivial. Then $\langle x \rangle = x/\omega(x)$ is computed \toprec{p^M} in $\tO(M \log p)$ bit operations. Assuming fast exponentiation, it takes $\tO(M^2 \log^2 p)$ bit operations to compute $\langle x \rangle^s$.\footnote{For $p$ odd, one could instead compute $x^t$ directly, with $t$ as in the proof of Proposition~\ref{prop:cfphis}.} Hence it takes $\tO(M^2 \log^2 p)$ bit operations to compute the value of $\phi_s(x)$ \toprec{p^M}. Combining this with the first complexity result completes the proof of the last statement.
\end{proof}

To use the Mahler expansion to compute values of a continuous function we need to compute binomials coefficients $\binom{s}{n}$ \toprec{p^M} for $\range{n}{0}{N-1}$. (We will also need these coefficients for the computation of measures and Iwasawa power series.) We use the following algorithm.

\begin{algo}[Computation of binomial coefficients]\label{algo:binom}\ 
\begin{description}\itemsep2pt
\item[\textbf{Input:}] $s \in \mathbb{Z}_p$.
\item[\textbf{Output:}] The binomial coefficients $\binom{s}{n}$, for $\range{n}{0}{N-1}$, \toprec{p^M}.
\item[\textbf{0.}] For $n = 1$ to $N$, do $v_n \gets v_p(n)$ and $u_n \gets (n p^{-v_n})^{-1} \smod p^M$.
\item[\textbf{1.}] Let $V$ be the largest integer $v \geq 0$ such that $p^v \leq N-1$. \\[1pt]
\hspace*{.0cm} Set $\ts \gets s \smod p^{M+V}$.
\item[\textbf{2.}] Set $A \gets 1$, $B \gets 0$, $b_0 \gets 1$.
\item[\textbf{3.}] For $n = 1$ to $N-1$, do
\begin{itemize}
\item[] If $\ts-n+1 = 0$, set $b_k \gets 0$ for $k = n$ to $N-1$ and go to Step 4.
\item[] $b \gets v_p(\ts-n+1)$, $a \gets (\ts-n+1) p^{-b} \smod p^M$,
\item[] $A \gets a u_n A \smod p^M$, $B \gets B + b - v_n$,
\item[] $b_n \gets A p^B \smod p^M$.
\end{itemize}
\item[\textbf{4.}] Return $\ldss{b_0}{b_{N-1}}$.
\end{description}
\end{algo}

\begin{remark} 
The precomputations in Step 0 need to be done only once for fixed $N$ and $M$. 
\end{remark}

\begin{lemma}\label{lem:compbin}
Algorithm~\ref{algo:binom} computes $\ldsss{\binom{s}{0}}{\binom{s}{1}}{\binom{s}{N-1}}$ \toprec{p^M\!} in $\tO(NM \log p)$ bit operations. 
\end{lemma}

\begin{proof}
For $n \geq 0$ and $x \in \mathbb{N}$ define 
\begin{equation*}
\mbinom{x}{n} = 
\begin{cases}
\;0                                   &\text{if $\binom{x}{n} = 0$}, \\[1\jot]
\displaystyle
\,\binom{x}{n} p^{-v_p(\binom{x}{n})} &\text{otherwise}.
\end{cases}
\end{equation*}
From the recurrence relation satisfied by binomial coefficients it follows that
\begin{equation*}
\mbinom{x}{n} = \frac{(x-n+1)p^{-v_p(x-n+1)}}{np ^{-v_p(n)}} \mbinom{x}{n-1}.
\end{equation*}
From this one can see by induction that at the end of $n$-th iteration of the loop in Step 3 we will have $A = \mbinom{\ts}{n} \smod p^M$ and $B = v_p(\binom{\ts}{n})$. Thus the algorithm returns $\binom{\ts}{n} \smod p^M$ for $\range{n}{0}{N-1}$. Now for $n \geq 1$ we can write $\binom{s}{n} = \frac{s}{n} \binom{s-1}{n-1}$ and therefore
\begin{equation*}
\left|\binom{s}{n}\right|_p \leq\, \frac{|s|_p}{|n|_p}.
\end{equation*}
It follows that $(1+T)^{p^{M+V}}\!\equiv 1 \tmod{p^M, T^N}$ and thus 
\begin{equation*}
(1+T)^s \equiv (1+T)^{\ts} \pmod{p^M, T^N}.
\end{equation*}
Therefore $\binom{s}{n} \equiv \binom{\ts}{n} \tmod{p^M}$ for $\range{n}{0}{N-1}$, and hence the algorithm returns the correct result.

Lastly we estimate the complexity of the algorithm. For an integer $n \geq 1$ we can compute $v_p(n)$ and $n p^{-v_p(n)}$ using $v_p(x) + 1$ divisions by $p$. Therefore Step 0 performs
\begin{equation*}
N + \lfloor N/p \rfloor + \lfloor N/p^2 \rfloor + \cdots \leq N/(p-1) 
\end{equation*}
divisions and therefore takes $\tO(N M \log p)$ bit operations. However, if we use the same method in Step 3 we may end up needing many divisions if $\ts-n+1$ has a very large $\p$-adic valuation. A better way to proceed is to use a divide-and-conquer algorithm, so that the computation of $a$ and $b$ can be done in $O(\log (M+V))$ divisions. The computation of $A$ and $b_n$ takes $\tO(M \log p)$ bit operations, the computation of $B$ is negligible, and, since $V \in O(\log(N))$, Step 3 takes $\tO(NM \log p)$ bit operations. This completes the proof.
\end{proof}

\begin{remark}
The algorithm can be improved in the following way. In Step 0 and Step 3 we can keep a counter, say \texttt{ct}, that we set to $p$ the first time we encounter an integer with a non-zero $\p$-adic valuation. Then at each step we decrease \texttt{ct} by $1$. If the value of \texttt{ct} is non-zero then the $\p$-adic valuation of the corresponding integer is zero. Otherwise we compute the valuation using the method explained above and reset \texttt{ct} to the value $p$. This gains a factor $p$ in computing the $\p$-adic valuation and the prime-to-$p$ part. It does not change the total computation cost estimate however. 
\end{remark}

Once we have computed sufficiently many Mahler coefficients of $f$ \toprec{p^M} we can use the algorithm below to compute values of $f$.

\begin{lemma}\label{lem:compvalues}
Let $f \in \mathcal{C}(\mathbb{Z}_p, \mathbb{Z}_p)$. Assume we have computed $f$ \toprec{p^M} with respect to $N \geq N_f(M)$. Then for all $s \in \mathbb{Z}_p$ we can compute $f(s)$ \toprec{p^M} in $\tO(N M \log p)$ bit operations. 
\end{lemma}

\begin{proof}
Let $\ldss{\tf_0}{\tf_{N-1}}$ be the first $N$ Mahler coefficients of $f$ \toprec{p^M}. Then
\begin{equation*}
f(s) \equiv \sum_{n=0}^{N-1} \tf_n \binom{s}{n} \pmod{p^M}.
\end{equation*}
The binomial coefficients are computed using Algorithm~\ref{algo:binom} in $\tO(N M \log p)$ bit operations,
and the computation of the sum also takes $\tO(N M\log p)$ bit operations.
\end{proof}

Another reason to compute Mahler coefficients is for approximating integrals.
\begin{lemma}\label{lem:compint}
Let  $f \in \mathcal{C}(\mathbb{Z}_p, \mathbb{Z}_p)$ and let $\mu \in \mathcal{M}(\mathbb{Z}_p, \mathbb{Z}_p)$.
Assume we have computed $f$ \toprec{p^M} with respect to $N \geq N_f(M)$ and have computed $\mu$ \toprec{(p^M, T^N)}. Then we can compute $\int f d\mu$ \toprec{p^M} in $\tO(N M \log p)$ bit operations. 
\end{lemma}

\begin{proof}
Write $\ldss{\tF_0}{\tF_{N-1}}$ (respectively $\ldss{\tf_0}{\tf_{N-1}}$) for the first $N$ coefficients of $F_\mu(T)$ (respectively Mahler coefficients of $f$) computed \toprec{p^M}. We have
\begin{equation*}
\int f \, d\mu \equiv \sum_{n=0}^{N-1} \tf_n \tF_n \pmod{p^M}.
\end{equation*}
This computation takes $\tO(N M \log p)$ bit operations.
\end{proof}

From these results it is obvious that having the best possible upper bounds on $N_f$ is crucial for getting the best complexity estimates. In the next subsection we consider this problem for the functions that interest us.

\subsection{Analyticity and Mahler coefficients}\label{subsec:analyticity}

A power series in $\mathbb{C}_p[[X]]$ is \emph{restricted} if its coefficients tend to zero or, equivalently, if it converges on $\mathbb{O}_p := \{x \in \mathbb{C}_p$ such that $|x|_p \leq 1\}$. Let $f : \mathbb{Z}_p \to \mathbb{C}_p$ be a function. We say $f$ is \emph{analytic} if there exists a restricted power series $\hat{f}(X) \in \mathbb{C}_p[[X]]$ such that 
\begin{equation*}
\vdepth{8pt}
f(x) = \hat{f}(x) \rlap{\qquad\text{for all $x \in \mathbb{Z}_p$}.}
\end{equation*}
For $h \geq 0$ we say $f$ is \emph{locally analytic of order $h$} if there exist restricted power series $\hat{f}_{a,h}(X)$, with $0 \leq a \leq p^h-1$, such that 
\begin{equation}\label{eq:locanal}
f(x) = \hat{f}_{a,h}((x-a)p^{-h})
\rlap{\qquad\text{for all $x \in a + p^h \mathbb{Z}_p$}.}
\end{equation}
Note that the series $\hat{f}_{a,h}$ are uniquely defined. An analytic function is therefore a locally analytic function of order $0$ and one can verify that, if $f$ is analytic of order $h_0$, then it is analytic of order $h$ for any $h \geq h_0$. Clearly a locally analytic function is continuous. We will see below that the fact that a continuous function is locally analytic has some important consequences for the rate of convergence to zero of its Mahler coefficients.\footnote{The main reference for these results is the article of Amice \cite{amice}; see also \cite{colmez:online} for a more accessible presentation.} The norm of a restricted power series $\hat{f}$, denoted $\|\hat{f}\|_{p,\infty}$, is defined as the maximum of the absolute values of its coefficients. It can be computed thanks to the following (see \cite[Prop.~1, \S6.1.4]{robert}):
\begin{equation}\label{eq:normal}
\|\hat{f}\|_{p,\infty} = \max_{x \in \mathbb{O}_p^\times} |\hat{f}(x)|_p.
\end{equation}
Let $f$ be a locally analytic function of order $h$. We define the $h$-norm of $f$ by 
\begin{equation*}
M_h(f) := \max_{0 \leq a \leq p^h-1} \|\hat{f}_{a,h}\|_{p,\infty}.
\end{equation*}
It follows from \eqref{eq:locanal} and \eqref{eq:normal} that 
\begin{equation}\label{eq:computemh}
M_h(f) = \max_{\substack{0 \leq a \leq p^h-1 \\ x \in \mathbb{O}_p^\times}} |f(a+p^h x)|_p.
\end{equation}

\begin{theorem}[Amice]\label{th:amice}
Let $f$ be a function in $\mathcal{C}(\mathbb{Z}_p, \mathbb{C}_p)$ with Mahler coefficients $(f_n)_{n \geq 0}$. Then $f$ is locally analytic of order $h$ if and only if 
\begin{equation*}
\left|\frac{f_n}{\lfloor n/p^h\rfloor!}\right|_p \to 0.
\end{equation*}
Moreover, if $f$ is locally analytic of order $h$ then for all $n \geq 0$ we have
\begin{equation*}
|f_n|_p \leq  M_h(f) \, \big|\lfloor n/p^h\rfloor!\big|_p.
\end{equation*}
\end{theorem}
\begin{proof}
The first assertion is Corollaire 2(b) of \cite[p.~162]{amice}. For the second assertion we use the notation and results of Corollaire 1 of \cite[p.~157]{amice}. Since the sequence $u_n = n$ is very well distributed (\emph{très bien répartie}), it follows from Corollaire 1(c) that
\begin{equation*}
v_p(n!/s_{n,h}) = \sum_{k \geq 1} \lfloor n/p^k \rfloor - \sum_{k = 1}^h \lfloor n/p^k \rfloor = v_p(\lfloor n/p^h\rfloor!).
\end{equation*}
Using Corollaire 1(b), we find that $\left(\lfloor n/p^h\rfloor! \binom{x}{n} \right)_{n \geq 0}$ is a normal basis (\emph{base normale}) of the Banach space of locally analytic functions of order $h$. Hence $M_h(f) = \sup_{n \geq 0} | f_n /\lfloor n/p^h\rfloor! |_p$, and the result follows.
\end{proof}

We are interested in finding optimal upper bounds for $N_{\phi_s}$. We remark that the function $\phi_s$ is locally analytic of order $1$ if $p$ is odd and of order $2$ if $p = 2$. Indeed, for $x \in  a + q\mathbb{Z}_p$ with $a \in \mathbb{Z}_p^\times$ we have
\begin{equation}\label{eq:diamondanalytic}
\langle x \rangle^s = \langle a\rangle^s \left(\frac{x-a}{a} + 1\right)^s = \langle a\rangle^s \sum_{n \geq 0} \binom{s}{n} \frac{q^n}{a^n} \left(\frac{x-a}{q}\right)^n.
\end{equation} 
In this way Theorem~\ref{th:amice} provides bounds on the Mahler coefficients of $\phi_s$. But we can do better with the following result.

\begin{proposition}\label{prop:cfphis}
Fix $s \in \mathbb{Z}_p$ and let $(\phi_n)_{n \geq 0}$ be the Mahler coefficients of $\phi_s$. Then for all $n \geq 0$ we have
\begin{equation*}
|\phi_n|_p \leq 
\begin{cases}
2^{-\lfloor n/2 \rfloor + 1} & \text{ if } p = 2, \\
|n!|_p & \text{ if $p$ is odd}.
\end{cases}
\end{equation*}
\end{proposition}

\begin{remark}
This result is close to optimal for odd $p$. Indeed, it implies that $|\phi_n/n!|_p \leq 1$ for all $n \geq 0$. On the other hand, we know this quantity is also bounded from below; otherwise Theorem~\ref{th:amice} would imply that $\phi_s$ is analytic, and in general it is not. 
\end{remark}

\begin{proof}
Assume $p$ is odd and let $B$ be a positive integer. By the Chinese Remainder Theorem we can find a positive integer $t$ such that 
\begin{equation*}
\begin{cases}
t \equiv s \pmod{p^B}, \\
t \equiv 0 \pmod{p-1}, \\
t \geq B.
\end{cases}
\end{equation*}
Then $x^t = \omega(x)^t \langle x\rangle^t \equiv \langle x\rangle^s \tmod{p^B}$ for $x \in \mathbb{Z}_p^\times$ and $x^t \equiv 0 \tmod{p^B}$ for $x \in p\mathbb{Z}_p$. Hence $|\phi_s(x) - x^t|_p \leq p^{-B}$ for all $x \in \mathbb{Z}_p$. It follows from \eqref{eq:normfromcoeff} and Theorem~\ref{th:amice} that
\begin{equation*}
|\phi_n|_p \leq \max\bigl\{|n!|_p,\ p^{-B}\bigr\}
\end{equation*}
and we obtain the result by taking $B$ large enough.

For the case $p = 2$ a similar proof works, provided $s$ is even. But for odd $s$ we need to use another approach. So we assume $s \in 1+2\mathbb{Z}_2$. Let $B$ and $t$ be positive integers such that $s \equiv t \tmod{2^B}$. As above we have
\begin{equation}\label{eq:case2}
\vdepth{8pt}
|\phi_s(x) - \omega(x) x^t|_2 \leq 2^{-B}
\end{equation}
where $\omega(x) := 0$ if $x \in 2\mathbb{Z}_2$. We turn now to the computation of bounds on the Mahler coefficients $(a_n)_{n \geq 0}$ of $\omega(x) x^t$. Let $i$ be a fixed square root of $-1$ in $\bar{\mathbb{Q}}_2$. Then $x \mapsto (\pm i)^x$ are continuous functions\footnote{But they are not analytic functions.} on $\mathbb{Z}_2$ and 
\begin{align*}
\omega(x) & = \frac{i}{2} ((-i)^x - i^x) = \frac{i}{2} \sum_{n \geq 0} (i-1)^n (i^n - 1) \binom{x}{n}.
\end{align*}
Thus the Mahler coefficients $(w_n)_{n \geq 0}$ of $\omega$ satisfy 
\begin{equation*}
v_2(w_n) = 
\begin{cases}
+\infty & \text{ if } n \equiv 0 \pmod{4}, \\
n/2 & \text{ if } n \equiv 2 \pmod{4}, \\
(n-1)/2 & \text{ if } n \equiv 1,3 \pmod{4}.
\end{cases}
\end{equation*}
In particular, $v_2(w_n) \geq \lfloor n/2\rfloor$ for all $n \geq 0$. We now apply the following lemma.
\begin{lemma}
Let $f$ and $g$ be two continuous functions with Mahler coefficients $(f_n)_{n \geq 0}$ and $(g_n)_{n \geq 0}$.
Then the Mahler coefficients $(c_n)_{n \geq 0}$ of $fg$ satisfy 
\begin{equation*}
v_p(c_n) \geq \min_{0 \leq k \leq n} \Big(v_p(f_k) + \min_{n-k \leq m \leq n} v_p(g_{m})\Big).
\end{equation*}
\end{lemma}
\begin{proof}[Proof of the lemma]
Let $k, m \geq 0$ be two integers. Then
\begin{equation*}
\binom{x}{k}\binom{x}{m} = \sum_{n = \max(k,m)}^{k+m} \nu_n(k,m) \binom{x}{n}
\end{equation*}
for some $\nu_n(k,m) \in \mathbb{Z}$. Hence
\begin{align*}
f(x)g(x) & = \sum_{k,m \geq 0} f_k g_m \binom{x}{k}\binom{x}{m} = \sum_{n \geq 0} \left(\sum_{k=0}^n f_k \sum_{m=n-k}^n g_m \, \nu_n(k,m) \right)\binom{x}{n}
\end{align*}
and the result follows from the expression of $c_n$ that can be derived from this equality.
\end{proof}
We take $f(x) = \omega(x)$ and $g(x) = x^t$. Then
\begin{align*}
v_2(a_n) & \geq \min_{0 \leq k \leq n} \Big(v_2(w_k) + \min_{n-k \leq m \leq n} v_2(m!)\Big) = \min_{0 \leq k \leq n} \Big(v_2(w_k) + v_2((n-k)!)\Big) \\
& \geq \min_{0 \leq k \leq n} (\lfloor k/2\rfloor + \lfloor (n-k)/2\rfloor) \geq \lfloor n/2 \rfloor - 1
\end{align*}
and the result follows from this estimate, taking $B$ sufficiently large in \eqref{eq:case2} as before.
\end{proof}

\begin{corollary}\label{cor:Nphis}
For every positive integer $M$ we have
\begin{equation*}
N_{\phi_s}(M) \leq 
\begin{cases}
2M + 2 & \text{ if $p = 2$}, \\
p M & \text{ if $p$ is odd}. 
\end{cases}
\end{equation*}
\end{corollary}
\begin{proof}
The result is clear for $p = 2$. For $p$ odd, it is enough to prove that the $\p$-adic valuation of $(pM)!$ is at least $M$. But $v_p((pM)!) = \sum\limits_{k \geq 1} \lfloor pM/p^k \rfloor \geq M$, and the result follows.
\end{proof}

Recall that for $x \in \mathbb{Z}_p^\times$ with $\langle x \rangle \in 1 + p^e\mathbb{Z}_p$ we have set
\begin{equation*}
\mathcal{L}_u(x) := \frac{\log_p \langle x\rangle}{\log_p u}
\end{equation*}
where $u$ is a fixed topological generator of $1 + p^e\mathbb{Z}_p$. For an integer $\ell \geq 0$ we define a continuous function $\psi_\ell$ in $\mathcal{C}(\mathbb{Z}_p,\mathbb{Z}_p)$ by 
\begin{equation*}
\psi_\ell(x) := 
\begin{cases}
x^{-1} \dbinom{\mathcal{L}_u(x)}{\ell} & \text { if } \langle x\rangle \in 1+p^e\mathbb{Z}_p, \\
0 & \text{ otherwise}.
\end{cases}
\end{equation*}
We have $x^{-1} (1+S)^{\mathcal{L}_u(x)} = \sum_{\ell \geq 0} \psi_\ell(x) S^\ell$ if $\langle x\rangle \in 1+p^e\mathbb{Z}_p$. These functions appear in the construction of the Iwasawa power series and will play an important part in their computations.

\begin{proposition} 
The Mahler coefficients $(\psi_{\ell,n})_{n \geq 0}$ of $\psi_\ell$ satisfy 
\begin{equation*}
|\psi_{\ell,n}|_p \leq \frac{|\lfloor n/p^e \rfloor!|_p}{|\ell!|_p}
\end{equation*}
for all $n \geq 0$.
\end{proposition}
\begin{proof}
It is clear that the function $\psi_\ell$ is locally analytic of order $e$. The $e$-norm is $1/|\ell!|_p$, by \eqref{eq:computemh}. The result follows from Theorem~\ref{th:amice}.
\end{proof}

\begin{corollary}\label{cor:Npsil}
For every positive integer $M$ we have
\begin{equation*}
N_{\psi_\ell}(M) \leq p^e (pM + \ell).
\end{equation*}
\end{corollary}

\begin{proof}
It is enough to prove that $v_p((pM+\ell)!) \geq  M + v_p(\ell!)$. But this is clear from the facts that $v_p((a+b)!) \geq v_p(a!) + v_p(b!)$ for any two non-negative integers $a$ and $b$, and that $v_p((pM)!) \geq M$.
\end{proof}

\begin{remark}
This upper bound is in general quite far from optimal. For example, with $p = 3$, $e = 1$, and $\ell = 10$ the corollary gives an upper bound of $210$ for $M = 20$, but computations give $N_{\psi_{10}}(20) = 85$. As the complexity of the computation of Iwasawa power series depends heavily upon the size of $N_{\psi_\ell}(M)$\,---\,see for example Theorem~\ref{th:compiwas}\,---\,it is a good idea to precompute the values $N_{\psi_\ell}(M)$ for small $\ell$ and $M$ and to use these instead in the computations.
\end{remark}

\subsectb{Computation of $\p$-adic cone zeta functions}{Computation of p-adic cone zeta functions}\label{subsec:comzetcone}

From \eqref{eq:defF} we see that, once a cone decomposition has been computed,\footnote{This will be the topic of the next subsection.} the computation of the $\p$-adic twisted partial zeta functions, and in turn that of the $\p$-adic $L$-functions, boils down to the computation of $\p$-adic cone zeta functions. We now turn to this computation, but first explaining how to deal with the several zeta functions associated with different additive characters all at once. Define the étale algebra 
\begin{equation*}
\vdepth{8pt}
\mathcal{R} := \mathbb{Q}_p[X]/(X^{c-1} + \cdots + 1)
\end{equation*}
and an additive character $\Xi$ from $\mathbb{Z}_E$ to $\mathcal{R}$ by setting
\begin{equation}
\Xi(\alpha) := \eta^a
\end{equation}
for each $\alpha \in \mathbb{Z}_E$, where $\eta$ is the image of $X$ in $\mathcal{R}$ and $a$ is any positive integer such that $\alpha \equiv a \tmod{\mathfrak{c}}$. The next result is straightforward.

\begin{lemma}\label{lem:xxyyzz}
Let $T_\mathcal{R}$ be the trace of $\mathcal{R}/\mathbb{Q}_p$ and let $\alpha \in \mathbb{Z}_E$. Then
\begin{equation*}
T_\mathcal{R}(\Xi(\alpha)) = \sum_{\substack{\xi \in X(\mathfrak{c}) \\ \xi \ne 1}} \xi(\alpha). \tag*{\hspace{-1em}\qed}
\end{equation*}
\end{lemma}

Note that $T_\mathcal{R}$ is trivial to compute since it is $\mathbb{Q}_p$-linear and we have
\begin{equation*}
T_\mathcal{R}(\eta^a) = 
\begin{cases}
c - 1 & \text{if $c \div a$}, \\
-1    & \text{otherwise}.
\end{cases}
\end{equation*}
Let $C := C(\beta; \lambda_1, \dots, \lambda_g)$ be a $\mathfrak{c}$-admissible cone. For $N \geq 1$ we define
\begin{equation} \label{eq:defFNC}
F_N(C, \mathfrak{c}; T) := T_\mathcal{R} \biggl[A(C, \Xi) \sum_{k_1, \dots, k_g = 0}^{(N-1)d} (1 + T)^{\mathcal{N}(\beta + \underline{k} \cdot \underline{\lambda})} \prod_{i=1}^g B_{k_i, (N-1)d}\big(\Xi(\lambda_i)\big)\biggr]
\end{equation}
where $A(C, \Xi) := \Xi(\beta)/\prod_{i=1}^g (1 - \Xi(\lambda_i))$ and $T_\mathcal{R}$ is extended in the natural way to $\mathcal{R}[[T]]$. It follows from Lemma~\ref{lem:xxyyzz} and Theorem~\ref{th:main} that $F_N(C, \mathfrak{c}; T) \equiv F(C, \mathfrak{c}; T) \tmod{T^N}$. We will use the expression above to compute approximations of $F(C, \mathfrak{c}; T)$ and of its associated measure $\mu_{p,C}^\mathfrak{c}$. We define 
\begin{equation}\label{eq:intzpC}
\mathcal{Z}_p(C, \mathfrak{c}; s) := \int \phi_{-s}(x) \, d\mu_{p,C}^\mathfrak{c}.
\end{equation}

\smallskip

We now determine some computation costs. These results, or at least their proofs, will be useful later to estimate the complexity of the computation of $\p$-adic $L$-functions; see Subsection~\ref{subsec:compL}. The first step is the computation of values of the rational functions $B_{k,K}$.

\begin{proposition}\label{prop:bk}
If $K$ is a non-negative integer then
\begin{equation*}
B_{0,K}(x) = x \left(\frac{x}{x-1}\right)^K - x + 1
\end{equation*}
and for $\,0 \leq k < K$ we have the recurrence formula 
\begin{equation*}
B_{k+1,K}(x) = x \left[(-1)^{k+1} \binom{K+1}{k+1} \left(\frac{x}{x-1}\right)^K + B_{k,K}\right].
\end{equation*}
\end{proposition}

\begin{proof}
We first establish another expression for the $B_{k,K}(x)$'s. 

\begin{lemma}
For $k \geq 0$ let $\,\Coeff_k$ denote the linear map that sends a polynomial in $\mathbb{Q}(x)[X]$ to the coefficient of its monomial of degree $k$. Then
\begin{equation*}
B_{k,K}(x) = \left(\frac{-x}{x-1}\right)^k \Coeff_k \left[\frac{\left(\dfrac{x}{x-1}+X\right)^{K+1}\!-1}{\dfrac{1}{x-1}+X}\right].
\end{equation*}
\end{lemma}
\begin{proof}[Proof of the lemma]
We compute
\begin{equation*}
B_{k,K}(x) = \left(\frac{-x}{x-1}\right)^k \sum_{n=k}^K \binom{n}{k} \left(\frac{x}{x-1}\right)^{n-k}
= \left(\frac{-x}{x-1}\right)^k \sum_{n=k}^K  \Coeff_k \left[\left(\frac{x}{x-1} + X\right)^n\right]
\end{equation*}
and the conclusion follows by evaluating the sum.
\end{proof}

Now define
\begin{equation*}
A_K(X) := \biggl( \biggl(\dfrac{x}{x-1}+X\biggr)^{K+1}\!- 1 \biggr) \!\Biggm/\! \biggl( \dfrac{1}{x-1}+X \biggr)
\end{equation*}
and let $a_k$ denote the coefficient of $X^k$ in $A_K(X)$. Then $B_{0,K} = A_K(0)$, which gives the first assertion. Since
\begin{equation*}
\vdepth{8pt}
A_K(X) \biggl(\dfrac{1}{x-1}+X\biggr) = \biggl(\dfrac{x}{x-1}+X\biggr)^{K+1}\!- 1
\end{equation*}
we see that
\begin{equation*}
\vdepth{16pt}
\frac{a_{k+1}}{x-1} + a_k = \binom{K+1}{k+1} \biggl(\frac{x}{x-1}\biggr)^{K-k}
\end{equation*}
for $0 \leq k < K$. The second assertion follows by induction.
\end{proof}

\begin{corollary}
Let $\alpha \in \mathbb{Z}_E$, coprime with $\mathfrak{c}$, and let $K \geq 1$ be an integer. Then we can compute $\ldss{B_{0,K}(\Xi(\alpha))}{B_{K,K}(\Xi(\alpha))} \in \mathcal{R}$ \toprec{p^M} in $\tO(K M c \log p+d \log c)$ binary operations. 
\end{corollary}

\begin{proof}
We apply Lemma~\ref{lem:compbin} to precompute the binomial coefficients $\ldss {\binom{K+1}{0}} {\binom{K+1}{K+1}}$ in $\tO(K M \log p)$ bit operations. Since the prime ideal $\mathfrak{c}$ is of degree $1$ there exist integers $\ldss{t_1}{t_d} \in \{\ldss{0}{c-1}\}$ such that $\theta_i \equiv t_i \tmod{\mathfrak{c}}$ for $\range{i}{1}{d}$. We assume $\ldss{t_1}{t_d}$ have been precomputed (using, say, \cite[Algo.~1.4.12]{cohen:book2}). Write $\alpha := a_1 \theta_1 + \cdots + a_d \theta_d \in \mathbb{Z}_E$. Then $\alpha \equiv a \tmod{\mathfrak{c}}$, where $a := a_1 t_1 + \cdots + a_d t_d \mod c$ is computed in $\tO(d \log c)$ bit operations. Next we precompute $(\eta^a/(\eta^a-1))^K$ in $\tO(M c \log K \log p)$ bit operations. From this we get $B_{0,K}(\eta^a)$ at negligible additional cost. Then, using the induction formula and the precomputed values, the computation of $B_{k+1,K}(\eta^a)$ from $B_{k,K}(\eta^a)$ takes only $\tO(Mc \log p)$ bit operations. 
\end{proof}

We now can give our first estimate. As we want our results to be valid for several different cones at once, we will express these estimates using $d$ and not $g$ (the number of generators), using the fact that $g \leq d$.

\begin{theorem}\label{th:compmeas}
For any positive integers $M$ and $N$ we can compute the measure $\mu_{p,C}^\mathfrak{c}$ \toprec{(p^M, T^N)} in $\tO(N^{d+1}d^{d+3} Mc \log p)$ bit operations. 
\end{theorem}

\begin{proof}
We compute $F(C; T)$ using \eqref{eq:defFNC}. Applying the previous proposition we precompute $B_{k,(N-1)d}(\Xi(\lambda_i))$, for $\range{k}{0}{(N-1)d}$ and $\range{i}{1}{d}$, in $\tO(Nd^2 Mc \log p)$ binary operations. With these values precomputed each computation of the inner product in \eqref{eq:defFNC} takes $\tO(d Mc \log p)$ bit operations. Now let $a := \mathcal{N}(\beta + \underline{k} \cdot \underline{\lambda})$. We compute $(1+T)^a$ \toprec{(p^M, T^N)}, using Lemma~\ref{lem:compbin}, in $\tO(N M\log p)$ bit operations, after having computed $a$ to the precision of $p^{M+V}$ (in the notation of Lemma~\ref{lem:compbin}). The main part of the computation of $a$ is the computation of the norm, which boils down to the computation modulo $p^{M+V}$ of the determinant of a $d \times d$ matrix; see \cite[\S 4.3]{cohen:book1}. This takes $\tO(d^3 (M+V)\log p)$ bit operations. The multiplication of the power series with the inner product takes $\tO(N Mc \log p)$ bit operations. The sum has $O(N^d d^d)$ terms so the computation of the sum requires $\tO(N^d d^d (N + d) Mc \log p)$ bit operations. The multiplication by $A(C, \Xi)$ and the computation of the trace take negligible time compared to the computation of the sum. The conclusion follows by putting everything together and simplifying.
\end{proof}

\begin{remark}
It is a good idea to retain the values of $B_{k,K}(\eta^a)$ in order to reuse them if several generators have the same image under $\Xi$. This does not affect the estimate of the cost of the computation since that estimate is dominated by the cost of computing the sum.
\end{remark}

Once the measure $\mu_{p,C}$ has been computed we can use it to compute values of $\mathcal{Z}_p(C, \mathfrak{c}; s)$.
  
\begin{corollary}\label{cor:compzc}
With a precomputation of cost $\tO(p^{d+1} d^{d+3} M^{d+2}c)$ bit operations, for any given $s$ in $\mathbb{Z}_p$ we can compute $\mathcal{Z}_p(C, \mathfrak{c}; s)$ \toprec{p^M} in $\tO(p^2M^3)$ bit operations.
\end{corollary}

\begin{proof}
We use the integral expression \eqref{eq:intzpC} for $\mathcal{Z}_p(C, \mathfrak{c};s)$. For this we need to compute $\mu_{p,C}^\mathfrak{c}$ \toprec{(p^M, T^N)} with $N > N_f(\phi_s)$. By Corollary~\ref{cor:Nphis} we can take $N = pM+2$; the cost of the precomputation comes from the previous theorem. To perform the integration we need to compute the first $N$ ($ = pM$) Mahler coefficients of $\phi_s$; by Lemma~\ref{lem:costphiscoef} the cost is $\tO(p^2M^3)$. The cost of computing the integral, as given by Lemma~\ref{lem:compint}, is $\tO(pM^2)$.
\end{proof}

Recall that $e$ is the largest positive integer such that $W_{p^e} \subset E(W_q)$ and that $u$ is a fixed topological generator of $1 + p^e \mathbb{Z}_p$. Denote by $\mathfrak{I}_p(C, \mathfrak{c}; X)$ the Iwasawa power series of $\mathcal{Z}_p(C, \mathfrak{c}; s)$, that is, the (unique) power series in $\mathbb{Z}_p[[X]]$ such that $\mathcal{Z}_p(C, \mathfrak{c}; 1-s) = \mathfrak{I}_p(C, \mathfrak{c}; u^s-1)$ for all $s \in \mathbb{Z}_p$. It is easy to adapt the proof of Theorem~\ref{thm:iwasawaseries} to show that this power series exists.\footnote{Actually, the definition of $\mathfrak{I}_p(C, \mathfrak{c}; X)$ is given in the proof of the next theorem.} We now give an explicit formula for $\mathfrak{I}_p(C, \mathfrak{c}; X)$ modulo $(p^M, X^L)$.

\begin{theorem}\label{th:compiwas}
Let $L$ and $M$ be positive integers and let $K = (p^e(pM+L)-1)d$.  Then
\begin{align*}
& \mathfrak{I}_p(C, \mathfrak{c}; X) \equiv \\
& T_\mathcal{R} \biggl[A(C, \Xi) \sum_{k_1, \dots, k_g = 0}^K \mathcal{N}(\beta + \underline{k} \cdot \underline{\lambda})^{-1}(1+X)^{\mathcal{L}_u(\mathcal{N}(\beta + \underline{k} \cdot \underline{\lambda}))}\prod_{i=1}^g B_{k_i, K}\big(\Xi(\lambda_i)\big)\biggr] \tmod{p^M, X^L}.
\end{align*}
Hence we can compute $\mathfrak{I}_p(C, \mathfrak{c}; X)$ \toprec{(p^M, X^L)} in $\tO(p^{ed}d^{d+3} (pM + L)^d M^2 Lc)$ bit operations.
\end{theorem}

\begin{proof} We start with a useful lemma.

\begin{lemma}\label{lem:diracdec}
Let $f$ be a continuous function on $\mathbb{Z}_p$ with Mahler coefficients $(f_n)_{n \geq 0}$. Let $M$ and $N$ be integers with $M \geq 1$ and $N \geq N_f(M)$. Let $\mu$ be a measure of norm $\leq 1$. Assume there exist a finite set $\mathcal{A}$ of elements of $\mathbb{Z}_p$ and an element $c_a$ in $\mathbb{O}_p$ for each $a \in \mathcal{A}$ such that
\begin{equation*}
F_\mu(T) \equiv \sum_{a \in \mathcal{A}} c_a (1+T)^a \pmod{T^N}.
\end{equation*}
Then
\begin{equation*}
\biggl|\int f \, d\mu - \sum_{a \in \mathcal{A}} c_a f(a)\biggr|_p \leq p^{-M}.
\end{equation*}
\end{lemma}

\begin{proof}[Proof of the lemma]
For each $a \in \mathbb{Z}_p$ let $\delta_a$ be the Dirac measure at $a$, and define the measure $\tmu := \mu - \sum_{a \in \mathcal{A}} c_a \delta_a$. Then
\begin{equation*}
\int g \, d\tmu = \int g \, d\mu - \sum_{a \in \mathcal{A}} c_a g(a)
%%% .
\end{equation*}
for all $g \in \mathcal{C}(\mathbb{Z}_p, \mathbb{C}_p)$. From this and the facts that $\|\mu\|_p \leq 1$ and $|c_a|_p \leq 1$ for all $a \in \mathcal{A}$ we see that $\tmu$ has norm $\leq 1$. By Lemma~\ref{lem:dirac} the associated power series is divisible by $T^N$, and thus
\begin{equation*}
\biggl|\int f \, d\tmu\biggr|_p = \Biggl| \sum_{n \geq N} f_n \int \binom{x}{n} \, d\tmu\Biggr|_p\leq \sup_{n \geq N} |f_n|_p \leq p^{-M}. \qedhere
\end{equation*}
\end{proof}

We apply the lemma repeatedly with $\mu = \mu_{p,C}^\mathfrak{c}$ and $f = \psi_\ell$ for each $\range{\ell}{0}{L}$ and using \eqref{eq:defFNC} for the definitions of the set $\mathcal{A}$ and the coefficients $c_a$. By Corollary~\ref{cor:Npsil} we can take $N = p^e(pM + L)$. We have 
\begin{multline*}
\mathfrak{I}_p(C, \mathfrak{c}; X) \;=\; \sum_{\ell \geq 0} \int \psi_\ell(x) \, d\mu_{p, C}^\mathfrak{c} \, X^\ell \;\equiv\; \sum_{\ell = 0}^{L-1} \int \psi_\ell(x) \, d\mu_{p, C}^\mathfrak{c} \, X^\ell \tmod{X^L} \\
\;\equiv\; T_\mathcal{R} \Biggl[\;\sum_{\ell = 0}^{L-1} A(C, \Xi) \sum_{k_1, \dots, k_g = 0}^{(N-1)d} \psi_\ell(\mathcal{N}(\beta + \underline{k} \cdot \underline{\lambda})) \prod_{i=1}^g B_{k_i, (N-1)d}\big(\Xi(\lambda_i)\big) X^\ell\;\Biggr] \tmod{p^M\!, X^L},
\end{multline*}
the expression for $\mathfrak{I}_p(C, \mathfrak{c}; X)$ coming from the fact that $\smash[b]{\sum\limits_{\ell = 0}^{L-1}} \psi_\ell(x) X^\ell \equiv x^{-1} (1+X)^{\mathcal{L}_u(x)} \tmod{X^L}$.

We now estimate the cost of computing $\mathfrak{I}_p(C, \mathfrak{c}; X)$ by this formula. As above, precomputation of the $B_{k,K}$'s has a cost of $\tO(p^{e+1} d M^2 Lc)$ binary operations, after which we can compute each inner product in $\tO(d Mc \log p)$ bit operations. The computation of $a := \mathcal{L}_u(\mathcal{N}(\beta + \underline{k} \cdot \underline{\lambda}))$ \toprec{p^M} takes $\tO(M(M + e) \log p + d^3 (M+V) \log p)$ bit operations.\footnote{See the proof of Theorem~\ref{th:compmeas} for the computation of the norm.} Once the norm has been computed the main computation is that of the $\p$-adic logarithm, which must be done \toprec{p^{M+e}} since we will be dividing by $\log_p(u) \in p^e \mathbb{Z}_p$. Since $\langle \mathcal{N}(\beta + \underline{k} \cdot \underline{\lambda}) \rangle \in 1 +p^e\mathbb{Z}_p$ this computation can be done using at most $M$ multiplications of precision $p^{M+e}$. Computing $(1+X)^a$ \toprec{(p^M, X^L)} takes $\tO(ML \log p)$ operations and multiplying by the inner product and the inverse of the norm costs $\tO(MLc \log p)$ bit operations. The result follows from simplifying and noting that the sum has $O(p^{ed}d^d(pM+L)^d)$ terms.\footnote{The costs of computing the product by $A(C, \Xi)$ and the trace are negligible compared to that of the sum.} 
\end{proof}

\begin{corollary}\label{cor:evaliwas}
With a precomputation of cost $\tO(p^{(e+1)d} d^{d+3} M^{d+2} c)$ bit operations, for any given $s$ in $\mathbb{Z}_p$ we can compute $\mathcal{Z}_p(C, \mathfrak{c}; s)$ \toprec{p^M} in $\tO(M^2 \log^2 p)$ bit operations. 
\end{corollary}

\begin{proof}
With $L := \lceil M/e \rceil$ we precompute $\mathfrak{I}_p(C, \mathfrak{c}; X)$ \toprec{(p^M, X^L)}, using the theorem to estimate the cost. Given $s \in \mathbb{Z}_p$, we compute $t := u^{1-s} - 1 \in p^e \mathbb{Z}_p$ \toprec{p^M} in $\tO(M^2 \log^2 p)$ bit operations and compute $\mathfrak{I}_p(C, \mathfrak{c}; t) \equiv \mathcal{Z}_p(C, \mathfrak{c}; s) \tmod{p^M}$ in $\tO((M^2/e) \log p)$ operations. The corollary follows.
\end{proof}

\begin{remark}
From its definition it would seem more natural to compute $\mathfrak{I}_p(C, \mathfrak{c}; X)$ modulo $(X, p^e)^L$. In particular, it would be enough to compute values of $\mathcal{Z}_p(C, \mathfrak{c}; s)$. This implies that,
for $0 \leq \ell < L$, the coefficient of $X^\ell$ would have to be computed \toprec{p^{e(L-\ell)}}. By Corollary~\ref{cor:Npsil} we can replace $N = p^e(p \lceil L/e \rceil +L)$ with $N = \max_{0 \leq \ell < L} p^e(pe(L-\ell)+\ell) = p^{e+1} Le$ in the formula $K = (N-1)d$. It is clear that this does not give a significant improvement in the estimate of the computation time.
\end{remark}

We finish this subsection with a result on the direct computation of $\mathcal{Z}_p(C, \mathfrak{c}; s)$ for a given $s \in \mathbb{Z}_p$. 

\begin{theorem}\label{th:componeval}
If $s \in \mathbb{Z}_p$ then 
\begin{equation*}
\mathcal{Z}_p(C, \mathfrak{c}; s) \;\equiv\; T_\mathcal{R} \biggl[ A(C, \Xi) \sum_{k_1, \dots, k_g = 0}^{(pM+1)d}\mathcal{N}(\beta + \underline{k} \cdot \underline{\lambda})^{-s} \prod_{i=1}^g B_{k_i, (pM+1)d}\big(\Xi(\lambda_i)\big) \biggr] \tmod{p^M} 
\end{equation*}
and hence we can compute $\mathcal{Z}(C, \mathfrak{c}; s)$ \toprec{p^M} in $\tO(p^d d^{d+3} M^{d+2} c)$ bit operations.
\end{theorem}

\begin{proof}
We use Lemma~\ref{lem:diracdec} again with $\mu = \mu_{p,C}^\mathfrak{c}$ and $f(x) = \phi_{-s}(x)$. By Corollary~\ref{cor:Nphis} we can take $N = pM+2$. This establishes the formula. We estimate the computation cost as in Theorem~\ref{th:compmeas}, replacing the computation of $(1+T)^a$ by that of $a^{-s}$, with $a := \mathcal{N}(\beta + \underline{k}\cdot\underline{\lambda})$, and accounting for the difference in the number of terms in the sum.
\end{proof}

\subsection{Explicit cone decomposition}\label{subsec:conecomp}

The construction of the measure $\mu_{p, \mathfrak{m}}^{\mathfrak{a}, \mathfrak{c}}$ relies on the existence of a cone decomposition of $\mathfrak{a}$ modulo $\mathfrak{m}$. Such a construction exists by a result of Cassou-Noguès \cite{cn:zeta} (see Theorem~\ref{th:conedec}), based on the work of Shintani \cite{shintani}, but the proof is nonconstructive. For $d = 1$ the construction is trivial, and the case $d = 2$ has been well studied (see below). For $d = 3$ an explicit efficient decomposition is given in \cite{diaz-friedman:3}, but quantitative results on the number of discrete cones obtained at the end are missing. A general construction is given in \cite{colmez:residue}, but this construction relies on the existence of a set of units satisfying certain conditions and there does not appear to be any practical way to construct such a set. However, the construction is generalized in \cite{diaz-friedman:general} for any set of units of maximal rank, at the price of using signed cones rather than cones.\footnote{That is, cones together with a $\pm$ sign; the decomposition is obtained by removing the cones with a $-$ sign.} This does not cause any complication in the computations and the changes needed to use signed cones are straightforward.  
 
\begin{remark}\label{rk:choicec} 
The construction given in the present article assumes that we can always find a $\mathfrak{c}$-admissible cone decomposition. Although this is always possible in the case $d = 1$ and $d = 2$ (see below), it is not guaranteed by Theorem~\ref{th:conedec} in general. This is not really a problem however; one can always first construct the several cone decompositions needed, then choose $\mathfrak{c}$ so that these cone decompositions are $\mathfrak{c}$-admissible. This is how it is done in~\cite{cn:zeta}. 
\end{remark}

We start with the case $d = 1$, which is straightforward.
\begin{proposition}
Assume $d = 1$, that is, $E = \mathbb{Q}$. Let $f$ and $a$ be positive integers such that $\mathfrak{m} = f\mathbb{Z} \cdot \infty$ and $\mathfrak{a} = a\mathbb{Z}$ and let $b = a (a^{-1} \smod f)$.\footnote{In particular, by (H1), $q$ divides $f$.} Then a $\mathfrak{c}$-admissible cone decomposition of $\,\mathfrak{a}$ modulo $\mathfrak{m}$ is given by the unique cone $C(b; af)$. \qed
\end{proposition}

It is also not difficult to construct a cone decomposition in the quadratic case. Indeed, assume $d = 2$; then $E$ is a real quadratic field. For two linearly independent elements $\gamma_0$ and $\gamma_1$ of $E^+$ we define the \emph{half-open rational cone of $\mathfrak{a}$ modulo $\mathfrak{m}$} (or simply the \emph{rational cone} of $\mathfrak{a}$ modulo $\mathfrak{m}$) generated by $\gamma_0$ and $\gamma_1$ to be the following subset of $\mathfrak{a} \cap E^+$:
\begin{equation*}
\RC_\mathfrak{m}(\gamma_0, \gamma_1; \mathfrak{a}) := 
\left\{\,\text{$s\gamma_0 + t\gamma_1$ with $s, t \in \mathbb{Q}$, $0 < s$, $0 \leq t$}\,\right\}
\cap \mathfrak{a} \cap E_\mathfrak{m}. 
\end{equation*}
We go from half-open rational cones to discrete cones using the formula
\begin{equation}\label{eq:convcone}
\RC_\mathfrak{m}(\gamma_0, \gamma_1; \mathfrak{a}) = \smash[b]{\bigcup_{\alpha \in \PC_\mathfrak{m}(\gamma_0, \gamma_1; \mathfrak{a})}}\!C(\alpha; \gamma_0, \gamma_1) \rlap{\quad\textrm{(disjoint union)}}
\end{equation}
where 
\begin{equation*}
\vdepth{6pt}
\PC_\mathfrak{m}(\gamma_0, \gamma_1; \mathfrak{a}) := \left\{\,\text{$s\gamma_0 + t\gamma_1$ with $s, t \in \mathbb{Q}$, $0 < s \leq 1$, $0 \leq t < 1$}\,\right\} \cap \mathfrak{a} \cap E_\mathfrak{m}.
\end{equation*}
Thus from a finite family of disjoint half-open rational cones giving a set of representatives of $\mathfrak{a} \cap E_\mathfrak{m}$ modulo $U_\mathfrak{m}(E)$ we can get a $\mathfrak{c}$-admissible cone decomposition of $\mathfrak{a}$ modulo $\mathfrak{m}$, provided that the generators of the rational cones are in $\mathfrak{af} \setminus \mathfrak{c}$. Let $\epsilon_+$ be the generator of $U_+(E)$, the group of totally positive units of $E$, with $\epsilon_+^{(2)} > 1 > \epsilon_+^{(1)}$. Let $i_\mathfrak{m}$ be the index of $U_\mathfrak{m}(E)$ as a subgroup of $U_+(E)$; thus $\epsilon_\mathfrak{m} := \epsilon_+^{i_\mathfrak{m}}$ is a generator of $U_\mathfrak{m}(E)$. For any totally positive element $\gamma$ of $\mathfrak{af} \setminus \mathfrak{c}$ one can easily check that $\RC_\mathfrak{m}(\gamma, \epsilon_{\mathfrak{m}}\gamma; \mathfrak{a})$ is a set of representatives of $\mathfrak{a} \cap E_\mathfrak{m}$ modulo $U_\mathfrak{m}(E)$, and therefore we can construct from this rational cone a $\mathfrak{c}$-admissible cone decomposition by the formula above.\footnote{We will explain below in the proof of the next proposition how to compute the elements of $\PC_\mathfrak{m}(\gamma_0, \gamma_1; \mathfrak{a})$.} However, the number of points in $\PC_\mathfrak{m}(\gamma, \epsilon \gamma; \mathfrak{a})$ is of the order of $\epsilon_\mathfrak{m}^{(2)} = (\epsilon_+^{(2)})^{i_\mathfrak{m}}$ and therefore much too large for computations in general. We use instead the following algorithm, also used in \cite{RS1}, which is based on the work of Hayes \cite{hayes:brumer}. For two elements $\gamma_0$ and $\gamma_1$ of $E^+$ we set 
\begin{equation*}
b(\gamma_0, \gamma_1) := \big\lceil \gamma_0^{(1)}/\gamma_1^{(1)} \big\rceil 
\quad\text{and}\quad R(\gamma_0, \gamma_1) := -\gamma_0 + b(\gamma_0, \gamma_1) \, \gamma_1.
\end{equation*}

\begin{algo}[Computation of cone decomposition in degree $2$]\label{algo:compdec2}\ 
\begin{description}\itemsep2pt
\item[\textbf{Input:}] Ideal $\mathfrak{a}$ coprime to $\mathfrak{c}$ and $\mathfrak{m}$.
\item[\textbf{Output:}] A $\mathfrak{c}$-admissible cone decomposition of $\mathfrak{a}$ modulo $\mathfrak{m}$.
\item[\textbf{1.}] Compute $g \in \mathbb{N}$ and $h \in \mathbb{Z}_E$ such that $\mathfrak{af} = \mathbb{Z}g + \mathbb{Z}h$.
\item[\textbf{2.}] If $h^{(2)} < h^{(1)}$ then do $h \gets -h$. \\[1pt]
\hspace*{.0cm} Do $h \gets h + \lceil -h^{(1)}/g \rceil g$. Set $(g_0, g_1) \gets (g, h)$.
\item[\textbf{3.}] While $g_1^{(2)} < g_0^{(2)}$, do $(g_0, g_1) \gets (g_1, R(g_0, g_1))$.
\item[\textbf{4.}] If $g_0 \in \mathfrak{c}$ then do $(g_0, g_1) \gets (g_1, R(g_0, g_1))$.
\item[\textbf{5.}] Set $D \gets \emptyset$ and $g_\text{last} \gets g_0\epsilon_\mathfrak{m}$.
\item[\textbf{6.}] While $g_0 \ne g_\text{last}$, do
\begin{itemize}
\smallskip
\item[\textbf{6.1.}] If $g_1 \not\in \mathfrak{c}$ then do
\smallskip
\item[]\quad            $b_0 \gets g_0$, $b_1 \gets g_1$ and $(g_0, g_1) \gets (g_1, R(g_0, g_1))$,
\smallskip
\item[]              Else
\smallskip
\item[]\quad            $g_2 \gets R(g_0, g_1)$,
\smallskip
\item[]\quad            $b_0 \gets g_0$, $b_1 \gets g_2$ and $(g_0, g_1) \gets (g_2, R(g_1, g_2))$.
\smallskip
\item[\textbf{6.2.}] For $a \in \PC_\mathfrak{m}(b_0, b_1; \mathfrak{a})$, do $D \gets D \cup \{C(a; b_0, b_1)\}$.
\end{itemize}
\item[\textbf{7.}] Return $D$.
\end{description}
\end{algo}

\begin{proposition}\label{prop:quadcone}
Let $D_E$ be the discriminant of $E$ and let $\epsilon_+$ be the generator of the group $U_+(E)$ of totally positive units of $E$ such that $\epsilon_+ > 1$.\footnote{To simplify the expressions, we assume from now on that $E$ is embedded into $\mathbb{R}$ by the map $x \mapsto x^{(2)}$.} Then Algorithm~\ref{algo:compdec2} computes a $\mathfrak{c}$-admissible cone decomposition of $\mathfrak{a}$ modulo $\mathfrak{m}$ in $\tO\big(\mathcal{N}(\mathfrak{af})\sqrt{D_E} + \mathcal{N}(\mathfrak{f})^2 \, \epsilon_+ \log D_E\big)$ operations in $E$. Moreover, this cone decomposition contains $\tO\big(\mathcal{N}(\mathfrak{f})\,\epsilon_+ \log D_E\big)$ cones.
\end{proposition}

\begin{proof}
The pair $(g, h)$ constructed in Step 2 of the algorithm satisfies $\mathfrak{af} = \mathbb{Z}g + \mathbb{Z}h$, $g ^{(1)} > h^{(1)}$ and $1 < h^{(2)}/h^{(1)}$. From this pair we construct in the following steps a sequence $(\tgamma_n)_{n \geq 0}$ with $(\tgamma_0, \tgamma_1) := (g, h)$ and $\tgamma_{n+1} := R(\tgamma_{n-1}, \tgamma_n)$ for $n \geq 1$. One can prove that the elements of this sequence satisfy 
\begin{equation*}
\text{$(1)~\mathfrak{af} = \mathbb{Z}\tgamma_n + \mathbb{Z}\tgamma_{n+1}$, \enspace
      $(2)~\tgamma_n^{(1)} > \tgamma_{n+1}^{(1)}$, \enspace and \enspace
      $(3)~\tgamma_n^{(2)}/\tgamma_n^{(1)} < \tgamma_{n+1}^{(2)}/\tgamma_{n+1}^{(1)}$}.
\end{equation*}
One can also prove (see below) that there exists an integer $N \geq 1$ such that 
\begin{equation*}
\text{$(4)~\tgamma_{n-1}^{(2)} < \tgamma_n^{(2)}$ \enspace for all \enspace $n \geq N$}.
\end{equation*}
We let $\gamma_n := \tgamma_{N+n}$ for $n \geq 0$. This is the sequence that is computed after Step 3. The points $\gamma_n$ are successive points on the \emph{convexity polygon} of $\mathfrak{af}$ as defined in \cite{hayes:brumer}. We can also extend the sequence in the other direction to obtain a sequence $(\gamma_n)_{n \in \mathbb{Z}}$, infinite in both directions, containing all the points on the convexity polygon, and for which we still have $\gamma_{n+1} = R(\gamma_{n-1}, \gamma_n)$ for all $n \in \mathbb{Z}$. It will be necessary to ensure that $\gamma_0 \not\in \mathfrak{c}$. If $\gamma_0 \in \mathfrak{c}$, we iterate one more time in Step 4 to replace $\gamma_0$ by $\gamma_1$ (that is, replacing $N$ by $N+1$). Indeed, by (1) $\gamma_0$ and $\gamma_1$ cannot both be in $\mathfrak{c}$. We can assume from now on that $\gamma_0 \not\in \mathfrak{c}$. The group $U_+(E)$ acts on the set $\{\gamma_n, n \in \mathbb{Z}\}$; thus there exists an integer $P_0 \geq 1$ such that $\gamma_{n+P_0} = \epsilon_+ \gamma_n$ for all $n \geq 0$. Therefore, for any $n \in \mathbb{Z}$ the union of the (disjoint) rational cones $\ldss{\RC_\mathfrak{m}(\gamma_n, \gamma_{n+1}; \mathfrak{a})}{\RC_\mathfrak{m}(\gamma_{n+P-1}, \gamma_{n+P}; \mathfrak{a})}$, with $P := i_\mathfrak{m} P_0$, gives a set of representatives of $\mathfrak{a} \cap E_\mathfrak{m}$ modulo $U_\mathfrak{m}(E)$ with generators in $\mathfrak{af}$. However, although $\gamma_0$, and thus also $\gamma_P$, do not belong to $\mathfrak{c}$, it is possible that $\gamma_n \in \mathfrak{c}$ for some $n$ in the range $1 \leq n \leq P-1$. In that case, $\gamma_{n-1}$ and $\gamma_{n+1}$ do not lie in $\mathfrak{c}$, by (1), and we use in Step 6.1 the fact that $\RC_\mathfrak{m}(\gamma_{n-1}, \gamma_n; \mathfrak{a}) \cup \RC_\mathfrak{m}(\gamma_n, \gamma_{n+1}; \mathfrak{a}) = \RC_\mathfrak{m}(\gamma_{n-1}, \gamma_{n+1}; \mathfrak{a})$ to get rid of cones with $\gamma_n \in \mathfrak{c}$. We end up with a rational cone decomposition having generators suitable for constructing a $\mathfrak{c}$-admissible cone decomposition of $\mathfrak{a}$ modulo $\mathfrak{m}$ using \eqref{eq:convcone}.

We will now estimate the complexity of the algorithm and the number of cones obtained at the end. First we find an upper bound on $N$, the number of iterations in Step~3. Assume $1 \leq t < N-1$. Since $\tgamma_{t+1}^{(1)} < \tgamma_t^{(1)}$ and $\tgamma_{t+1}^{(2)} < \tgamma_t^{(2)}$ we have $\mathcal{N}(\tgamma_{t+1}) < \mathcal{N}(\tgamma_t)$, and since both are divisible by $\mathcal{N}(\mathfrak{af})$ we have $N \leq \mathcal{N}(h)/\mathcal{N}(\mathfrak{af}) + 1$. Now assume, without loss of generality, that $\sqrt{D_E}^{(1)} = \sqrt{D_E} > 0$ and $\sqrt{D_E}^{(2)} = - \sqrt{D_E} < 0$. Write $h = a + b \sqrt{D_E}$ with $a, b \in \frac{1}{2}\mathbb{Z}$. Since $h^{(2)}> h^{(1)}$, it follows that $b < 0$, and also $a > - b \sqrt{D_E} >  0$, as $h \in \mathbb{Z}_E^+$. On the other hand, $a < g - b\sqrt{D_E}$ because $h^{(1)} < g$. From the fact that $\mathfrak{af} = \mathbb{Z}g + \mathbb{Z}h$, we find that $b = -\mathcal{N}(\mathfrak{af})/2g$. We compute $\mathcal{N}(h) =  a^2 - b^2 D_E < (g - b\sqrt{D_E})^2 - b^2D_E = g^2 - 2gb\sqrt{D_E} = g^2 + \mathcal{N}(\mathfrak{af}) \sqrt{D_E} \in O(\mathcal{N}(\mathfrak{af})^2 \sqrt{D_E})$. Thus Step~3 requires $O(\mathcal{N}(\mathfrak{af}) \sqrt{D_E})$ operations in $E$. 

Next we estimate the size of $P = i_\mathfrak{m}P_0$. By \cite[p.~161]{hayes:brumer} we have
\begin{equation}\label{eq:bndP}
\log C + \log \frac{\epsilon_+^{(2)}}{\epsilon_+^{(2)} - \epsilon_+^{(1)}} + \frac{1}{2} \log D_E \geq P \frac{\log \epsilon_+^{(2)}}{\epsilon_+^{(2)}-\epsilon_+^{(1)}}
\end{equation}
where 
\begin{equation*}
\vdepth{14pt}
C := \frac{\epsilon_\mathfrak{m}^{(2)}-\epsilon_\mathfrak{m}^{(1)}}{\sqrt{D_E}} \inf_{n \in \mathbb{Z}} \frac{\mathcal{N}(\gamma_n)}{\mathcal{N}(\mathfrak{af})}.
\end{equation*}
For $n \in \mathbb{Z}$ we see that $b(\gamma_{n-1}, \gamma_n) = 2$ if and only if $\gamma_n$ is midway between $\gamma_{n-1}$ and $\gamma_{n+1}$, in which case we say $\gamma_n$ is a \emph{midpoint}. If $\gamma_n$ is not a midpoint we say $\gamma_n$ is a \emph{vertex}. If $\gamma_n$ is a vertex we have $\mathcal{N}(\gamma_n) \leq \mathcal{N}(\mathfrak{af}) \sqrt{D_E}$ (see \cite[Prop.~5.5]{hayes:brumer}) and it follows that $C \leq \epsilon_{\mathfrak{m}}^{(2)}$. Substitution in \eqref{eq:bndP} gives
\begin{equation*}
P \in O\big(\epsilon_+^{(2)} (\log \epsilon_\mathfrak{m}^{(2)} + \log D_E)\big) \subset \tO\big(i_\mathfrak{m} \epsilon_+^{(2)} \log D_E\big).
\end{equation*}

Lastly, we need to explain how to perform Step 6.2, we need to estimate the cost of the computation, and  we need to estimate the number of points in the sets $\PC_\mathfrak{m}(b_0, b_1; \mathfrak{a})$. There are two cases to consider, (I) $(b_0, b_1) = (\gamma_n, \gamma_{n+1})$ and (II) $(b_0, b_1) = (\gamma_n, \gamma_{n+2})$, for an arbitrary $n \in \mathbb{Z}$. We have the bijection
\begin{equation*}
\mathfrak{a}/(\mathbb{Z} b_0 + \mathbb{Z} b_1) \buildrel 1:1 \over \longrightarrow \left\{ s b_0 + t b_1 \textrm{ with } s, t \in \mathbb{Q},\ 0 < s \leq 1,\ 0 \leq t < 1 \right\} \cap \mathfrak{a},
\end{equation*}
defined as follows. For a given class $\bar\alpha$ in $\mathfrak{a}/(\mathbb{Z} b_0 + \mathbb{Z} b_1)$, lift $\bar\alpha$ to an arbitrary element $\alpha \in \mathfrak{a}$ and write $\alpha = s b_0 + t b_1$ with $s, t \in \mathbb{Q}$. Then the map above sends $\bar\alpha$ to 
\begin{equation*}
[ \alpha ]_{(b_0, b_1)} := (s - \lceil s \rceil + 1) b_0 + (t - \lfloor t \rfloor) b_1.
\end{equation*}
Moreover, since $b_0$ and $b_1$ lie in $\mathfrak{af}$ there is a well-defined map from $\mathfrak{a}/(\mathbb{Z} b_0 + \mathbb{Z} b_1)$ to $\mathfrak{a}/\mathfrak{af}$ which sends $\bar\alpha$ to the class of $\alpha$ modulo $\mathfrak{af}$. The set $\PC_\mathfrak{m}(b_0, b_1; \mathfrak{a})$ consists of precisely those elements $[\alpha]_{(b_0, b_1)}$ for which $\bar\alpha \in \mathfrak{a}/(\mathbb{Z} b_0 + \mathbb{Z} b_1)$ is congruent to $1$ modulo $\mathfrak{f}$. Since $\mathfrak{a}$ and $\mathfrak{f}$ are coprime this map is surjective and therefore $\PC_\mathfrak{m}(b_0, b_1; \mathfrak{a})$ contains exactly $d/\mathcal{N}(\mathfrak{f})$ elements, where $d := (\mathfrak{a}:\mathbb{Z} b_0 + \mathbb{Z} b_1)$. Using the methods of \cite[\S4.1.3]{cohen:book2}, two elements $\alpha_0$ and $\alpha_1$ of $\mathfrak{a}$ can be constructed with $\mathfrak{a} = \mathbb{Z} \alpha_0+\mathbb{Z}\alpha_1$ and $\mathfrak{a}/(\mathbb{Z} b_0 + \mathbb{Z} b_1) = (\mathbb{Z}/d_0\mathbb{Z})\bar\alpha_0 + (\mathbb{Z}/d_1\mathbb{Z})\bar\alpha_1$, where $d_0$ and $d_1$ are positive integers and $d := d_0d_1$. Thus the elements of $\mathfrak{a}/(\mathbb{Z} b_0 + \mathbb{Z} b_1)$ can easily be enumerated in $O(d)$ operations in $E$. For case (I) we have $\mathbb{Z} \gamma_n + \mathbb{Z} \gamma_{n+1} = \mathfrak{af}$ and therefore $d = \mathcal{N}(\mathfrak{f})$ and $\PC_\mathfrak{m}(\gamma_n, \gamma_{n+1}; \mathfrak{a})$ contains only one element. For case (II) we have 
\begin{equation*}
d = (\mathfrak{a} : \mathbb{Z}\gamma_n + \mathbb{Z}\gamma_{n+2}) = (\mathfrak{a} : \mathbb{Z}\gamma_n + \mathbb{Z}\gamma_{n+1}) (\mathbb{Z}\gamma_n + \mathbb{Z}\gamma_{n+1} : \mathbb{Z}\gamma_n + \mathbb{Z}\gamma_{n+2}) = \mathcal{N}(\mathfrak{f}) \, b(\gamma_n, \gamma_{n+1}).
\end{equation*}
If $\gamma_{n+1}$ is a midpoint we have $b(\gamma_n, \gamma_{n+1}) = 2$. We now need to estimate the size of $b(\gamma_n, \gamma_{n+1})$ when $\gamma_n$ is a vertex. Since $b(\gamma_n, \gamma_{n+1}) = b(\gamma_{n+P_0}, \gamma_{n+P_0+1})$ it is enough to look at what happens for the vertices among $\ldss{\gamma_0}{\gamma_{P_0-1}}$. Writing $b_n := b(\gamma_{n-1}, \gamma_n)$ to simplify the notation, we have by construction 
\begin{equation*}
\frac{\gamma_{n-1}^{(1)}}{b_n - 1} > \gamma_n^{(1)} > \frac{\gamma_{n-1}^{(1)}}{b_n} 
\end{equation*}
and therefore
\begin{equation*}
\vdepth{16pt}
\frac{\gamma_0^{(1)}}{(b_{P_0}-1)(b_{P_0-1}-1) \cdots (b_1-1)} > \gamma_{P_0}^{(1)} > \frac{\gamma_0^{(1)}}{b_{P_0} b_{P_0-1} \cdots b_1}. 
\end{equation*}
From the fact that $\gamma_{P_0} = \epsilon_+ \gamma_0$ we find that
\begin{equation}\label{eq:bndvertex}
\epsilon_+^{(2)} > \prod_{i=1}^{P_0} (b_i-1).
\end{equation}
The indices $i$ for which $b_i = 2$, that is, corresponding to the midpoints, do not contribute to the product. For indices corresponding to vertices we get
\begin{equation*}
\sum_{\substack{1 \leq i \leq P_0 \\ \gamma_i \text{ is a vertex }}} b_i \in O(\epsilon_+^{(2)}).
\end{equation*}
We now put everything together to get the result. For Step 1, we assume that an ideal is given by a $2{\times}2$ integral matrix expressing a basis of the ideal of the (fixed) integral basis of $E$. This step amounts to an HNF reduction of a $2{\times}4$ matrix (see \cite[\S 4.7.1]{cohen:book1}). Since we can reduce the entries of this matrix modulo $\mathcal{N}(\mathfrak{af})$ this step takes $\tO(\log(\mathcal{N}(\mathfrak{af})))$ bit operations and hence is negligible. Step 3 takes $O(\mathcal{N}(\mathfrak{af}) \sqrt{D_E})$ operations in $E$. The loop in Step 6 is iterated $P$ times. The most costly operation is in Step 6.2. The cost of computing $\alpha_0$, $\alpha_1$, $d_0$, and $d_1$ is essentially that of an SNF reduction of a $2{\times}2$ matrix with coefficients of size $\leq \mathcal{N}(\mathfrak{af})$ and hence can be neglected. Enumerating the elements of $\PC_\mathfrak{m}(b_0, b_1; \mathfrak{a})$ takes a total of $O(\mathcal{N}(\mathfrak{f})P)$ operations in $E$ for the midpoints and $O(i_\mathfrak{m} \mathcal{N}(\mathfrak{f}) \epsilon_+)$ for the vertices. This gives the estimate on the complexity of the algorithm. To count the cones we observe that the midpoints give $O(P)$ cones and the vertices give $O(i_\mathfrak{m} \epsilon_+)$ cones. The conclusion is established by using the fact that $i_\mathfrak{m} \in O(\mathcal{N}(\mathfrak{f}))$.
\end{proof}

\begin{remark} 
One could ask what would happen if we were first to construct a cone decomposition using the algorithm without any restriction related to $\mathfrak{c}$, that is, deleting Step 4 and always doing the first part in Step 6.1, then choosing the prime ideal $\mathfrak{c}$ so that the decomposition computed is $\mathfrak{c}$-admissible. In fact this would not change the complexity or even the order of the number of cones, since both are dominated by the contributions of the midpoints and, in fact, for these we get the same number of discrete cones in cases (I) and (II). On the other hand, this would probably force the norm of $\mathfrak{c}$ to get significantly larger and that would adversely affect the complexity of the remaining computations.
\end{remark}

\begin{remark}
The complexity and the estimate of the number of cones given by this proposition appear in practice to be very pessimistic as they are of the order of the exponential of $R_E$, the regulator of $E$, whereas computations point towards something of the size of $R_E$. Indeed, one can use \eqref{eq:bndvertex} to show that the number of vertices among $\ldss{\gamma_0}{\gamma_{P_+}}$ is $O(R_E)$. However, it appears difficult to bound the number of midpoints. One can prove that if $\gamma_n$ is a vertex then the number of midpoints following it is $\big\lfloor 1/(1- b_n + \gamma_{n-1}^{(1)}/\gamma_n^{(1)}) \big\rfloor$, and so this problem is related to the question of how close a quadratic irrationality can be to an integer. In order to bound more efficiently the number of cones one would need to bound the size of the $b_n$'s. This could be done for example using \eqref{eq:bndvertex} by finding some non-trivial lower bound on the number of vertices.
\end{remark}

\subsectb{Computations of $\p$-adic L-functions}{Computations of p-adic L-functions}\label{subsec:compL}

We use the results from the preceding subsections to estimate the complexity of computing $L$-functions. We will make certain assumptions. As noted above, we assume we have computed the necessary data to work in $E$: ring of integers, class group, units, etc. We assume also that we have at our disposal a prime ideal $\mathfrak{c}$ satisfying the hypotheses (H1), (H2), and (H3) and the additional hypothesis 
\begin{enumerate}
\item[(H4)] Either $\chi$ is non-trivial and $\chi(\mathfrak{c}) \ne 1$, or $\chi$ is trivial and $\langle c\rangle \not\in 1 + p^{e+1}\mathbb{Z}_p$.\footnote{Note that we always have $\langle c\rangle \in 1 + p^e\mathbb{Z}_p$ by Lemma~\ref{lem:normcong}.}
\end{enumerate}
We assume we have computed a list of integral ideals $\mathfrak{a}_i$, $\range{i}{1}{h_\mathfrak{m}(E)}$, coprime to $\mathfrak{c}$ and $\mathfrak{m}$ and representing all the classes of $\Cl_\mathfrak{m}(E)$. Finally, we assume we have computed a cone decomposition for each ideal $\mathfrak{a}_i$; we will denote by $B$ the maximum number of cones among these decompositions (see the previous subsection). In what follows $\delta$ will denote the degree of $\mathbb{Q}_p(\chi)/\mathbb{Q}_p$.

\begin{lemma}
Assume the ERH.  Then there exists a prime ideal $\mathfrak{c}$ satisfying hypotheses \textup{(H1)} through \textup{(H4)}, with $c \in O(\log^2 (\mathcal{N}(\mathfrak{f}) D_E))$ if $\chi$ is non-trivial and $c \in \tO(p^{2m_0} \log^2 (D_E))$ if $\chi$ is trivial, where $m_0 \geq 0$ is such that $\mathbb{Q}_{m_0} = E \cap \mathbb{Q}_\infty$. 
\end{lemma}

\begin{proof}
We use Theorem~1 of of \cite{bach}. For the case $\chi$ non-trivial the application is direct. For the case $\chi$ trivial we apply the theorem to the character $\rho$ generating the group of characters of $\Gal(E_1/E)$. The absolute norm of the conductor of $\rho$ divides the absolute norm of the conductor of $\mathbb{Q}_{m_0+1}/\mathbb{Q}_{m_0}$, the $\p$-adic valuation of which is $v_p(q) + (p^{m_0+1}-1)/(p-1)$. The result follows. 
\end{proof}

\begin{theorem}\label{thm:xyzxyz}
Let $M$ and $N$ be positive integers. Under the assumptions enumerated at the beginning of this subsection the measures $\mu_{p, \mathfrak{m}}^{\mathfrak{a}_i, \mathfrak{c}}$, for $\range{i}{1}{h_\mathfrak{m}(E)}$, can be computed \toprec{(p^M, T^N)} in $\tO(h_\mathfrak{m}(E) d^{d+3} B N^{d+1} M c \log p)$ bit operations. 
\end{theorem}

\begin{proof}
This follows directly from Theorem~\ref{th:compmeas}.
\end{proof}

\begin{corollary}\label{cor:compvalues1}
Let $M$ be a positive integer and let $s \in \mathbb{Z}_p$, with $s \ne 1$ if $\chi$ is trivial. Under the assumptions enumerated at the beginning of this subsection and after precomputations of cost $\tO(h_\mathfrak{m}(E) p^{d+1} d^{d+3} B M^{d+2}c)$ bit operations, two algebraic integers $\beta$ and $\gamma$, both belonging to $\mathbb{Z}[\chi]$ and with $\gamma L_{p, \mathfrak{m}}(\chi; s)$ lying in $\mathbb{Z}_p[\chi]$ and $|\gamma L_{p, \mathfrak{m}}(\chi; s) - \beta|_p \leq p^{-M}\!$, can be computed in $\tO(h_\mathfrak{m}(E) (p^2 M^3 + \delta M \log p))$ bit operations. Moreover, $p^{-1/(p-1)} \leq |\gamma|_p \leq 1$ if $\chi$ is non-trivial and $|\gamma|_p = p^{-e} |s-1|_p$ if $\chi$ is trivial.
\end{corollary}

\begin{proof}
We precompute the measures $\mu_{p, \mathfrak{m}}^{\mathfrak{a}_i, \mathfrak{c}}$, for $\range{i}{1}{h_\mathfrak{m}(E)}$, \toprec{(p^M, T^N)}, with $N := pM+2$, the computation cost being given by Theorem~\ref{thm:xyzxyz}. We let $\beta$ be an approximation of the sum in \eqref{def:padicL} \toprec{p^M}. The values of $\mathcal{Z}_{p,\mathfrak{m}}(\mathfrak{a}_i^{-1}, \mathfrak{c}; s)$, for $\range{i}{1}{h_\mathfrak{m}(E)}$, \toprec{p^M} are computed using Corollary~\ref{cor:compzc} \textit{mutatis mutandis} in $\tO(h_\mathfrak{m}(E) p^2M^3)$ operations. The rest of the computation of $\beta$ takes $O(h_\mathfrak{m}(E) \delta M \log p)$ bit operations. We now let $\gamma$ be an approximation of $\chi(\mathfrak{c}) \langle c\rangle^{1-s} - 1$ \toprec{p^M}. The computation of $\gamma$ takes $\tO(M^2 \log^2 p + \delta M \log p)$ bit operations, and it follows from \eqref{def:padicL} that $\gamma L_{p, \mathfrak{m}}(\chi; s)$ lies in $\mathbb{Z}_p[\chi]$ and 
$|\gamma L_{p, \mathfrak{m}}(\chi; s) - \beta|_p \leq p^{-M}$. The assertions concerning the absolute value of $\gamma$ are straightforward (see the proof of Corollary~\ref{cor:boundc}). 
\end{proof}

Recall that $E \cap \mathbb{Q}_\infty = \mathbb{Q}_{m_0}$ and $E(\mathfrak{m}) \cap \mathbb{Q}_\infty = \mathbb{Q}_{m_0+m_1}$, so that $e = m_0 + v_p(q)$. 

\begin{theorem}\label{thm:compiwa}
Let $M$ and $L$ be positive integers. Under the assumptions enumerated at the beginning of this subsection there exist polynomials $B(X)$ and $C(X)$ in $\mathbb{Z}_p[\chi][X]$, with a cost of $\tO(h_\mathfrak{m}(E)(p^{ed} d^{d+3} B (pM+L)^d M^2 Lc + \delta M \log p))$ bit operations to compute, such that 
\begin{equation*}
C(X) \mathfrak{I}_{p,\mathfrak{m}}(\chi; X) - B(X) \in 
\begin{cases}
p^M \mathbb{Z}_p[\chi][[X]] + X^L \mathbb{Z}_p[\chi][[X]]
& \parbox[t]{7em}{if $\chi$ is trivial or \\ not of type $W$,} \\
\dfrac{p^M}{X + \pi} \mathbb{Z}_p[\chi][[X]] + \dfrac{X^L}{X + \pi} \mathbb{Z}_p[\chi][[X]]
& \text{otherwise,}
\end{cases}
\end{equation*}
with $\pi \in \mathbb{O}_p$ satisfying
\begin{equation*}
\frac{1}{p^{m_1-1}(p-1)} \leq v_p(\pi) \leq \frac{1}{p-1}.
\end{equation*}
Moreover, $p^{-1/(p-1)} \leq |C(0)|_p \leq 1$ if $\chi$ is non-trivial and $C(0) = 0$, $|C'(0)|_p = 1$ if $\chi$ is trivial. 
\end{theorem}

\begin{proof}
We use the notation and results from the proof of Theorem~\ref{thm:iwasawaseries}. The polynomial $C(X)$ is an approximation modulo $(p^M, X^L)$ of the power series $C(\mathfrak{c}, \chi; X)$ and the polynomial $B(X)$ is approximation modulo $(p^M, X^L)$ of the power series 
\begin{equation*}
\sum_{i=1}^{h_\mathfrak{m}(E)} \chi(\mathfrak{a}_i^{-1}) N(\mathfrak{a}_i; X) A(\mathfrak{a}_i, \mathfrak{c}; X). 
\end{equation*}
The first assertion follows by \eqref{eq:compiwaseries} and the integral properties of $ \mathfrak{I}_{p,\mathfrak{m}}(\chi; X)$, with $\pi := (\xi-1)/\xi$. Since $\xi$ has order $p^m$ for some integer $m$ with $1 \leq m \leq m_1$ this proves the inequalities on the $\p$-adic valuation of $\pi$. The properties of $C(X)$ follow from (H4). We now evaluate the complexity of the computation of $B(X)$ and $C(X)$. Let $\mathfrak{a}$ be one of the ideals $\mathfrak{a}_i$ and let $\{C_1, \dots, C_m\}$ be a $\mathfrak{c}$-admissible cone decomposition of $\mathfrak{a}$. Then
\begin{equation*}
A(\mathfrak{a}, \mathfrak{c}; X) = \sum_{j=1}^m \mathfrak{I}_p(C_j, \mathfrak{c}; X) 
\end{equation*}
and the computation cost of $A(\mathfrak{a}, \mathfrak{c}; X)$ \toprec{(p^M, X^L)} follows from Theorem~\ref{th:compiwas}. The computation time of $C(X)$ is negligible compared to that of $B(X)$.
\end{proof}

\begin{corollary}\label{cor:compvalues2}
Let $M$ be a positive integer and let $s \in \mathbb{Z}_p$, with $s \ne 1$ if $\chi$ is trivial. Under the assumptions enumerated at the beginning of this subsection and after precomputations costing $\tO(h_\mathfrak{m}(E)(p^{(e+1)d} d^{d+3} B M^{d+2} Lc + \delta M \log p))$ bit operations, two algebraic integers, $\beta$ and $\gamma$,  both belonging to $\mathbb{Z}[\chi]$ and with $\gamma L_{p, \mathfrak{m}}(\chi; s)$ lying  in $\mathbb{Z}_p[\chi]$ and 
\begin{equation}\label{eq:compvalues2}
\left|\gamma L_{p, \mathfrak{m}}(\chi; s) - \beta\right|_p \leq
\begin{cases}
\,p^{-M}         & \text{if $\chi$ is trivial or not of type $W$}, \\[1\jot]
\,p^{-M+1/(p-1)} & \text{otherwise},
\end{cases}
\end{equation}
can be computed in $\,\tO(M^2 \log p\, (\delta/e + \log p))$ bit operations. Moreover, $p^{-1/(p-1)} \leq |\gamma|_p \leq 1$ if $\chi$ is non-trivial and $|\gamma|_p = p^{-e} |s-1|_p$ if $\chi$ is trivial. 
\end{corollary}

\begin{proof}
We precompute the polynomials $B(X)$ and $C(X)$ with $L := \lceil M/e \rceil$. The precomputation cost is given by Theorem~\ref{thm:compiwa}. Then we compute $t := u^{1-s} - 1$ \toprec{p^M} in $\tO(M^2 \log^2 p)$ bit operations, and compute $\beta$ to be $B(t)$, respectively $\gamma$ to be $C(t)$, \toprec{p^M} in $\tO(\delta M^2/e \log p)$ bit operations. The result follows from \eqref{eq:iwasawaseries} and the theorem.
\end{proof}

We conclude with the cost of computing a single value of a $\p$-adic $L$-function without precomputations. 

\begin{theorem}
Let $M$ be a positive integer and let $s \in \mathbb{Z}_p$, with $s \ne 1$ if $\chi$ is trivial. Under the assumptions enumerated at the beginning of this subsection, two algebraic integers $\beta$ and $\gamma$, both belonging to $\mathbb{Z}[\chi]$ and with $\gamma L_{p, \mathfrak{m}}(\chi; s)$ lying in $\mathbb{Z}_p[\chi]$ and $|\gamma L_{p, \mathfrak{m}}(\chi; s) - \beta|_p \leq p^{-M}$, can be computed at a cost of $\tO(h_\mathfrak{m}(E) (p^d d^{d+3} M^{d+2} c + \delta M \log p))$ bit operations. Moreover, $ p^{-1/(p-1)} \leq |\gamma|_p \leq 1$ if $\chi$ is non-trivial and $|\gamma|_p = p^{-e} |s-1|_p$ if $\chi$ is trivial. 
\end{theorem}

\begin{proof}
We proceed as in the proof of Corollary~\ref{cor:compvalues1} with the same definitions for $\beta$ and $\gamma$. The cost of computing $\gamma$ is the same. We construct $\beta$ by computing the values of $\mathcal{Z}_{p,\mathfrak{m}}(\mathfrak{a}_i^{-1}, \mathfrak{c}; s)$, for $\range{i}{1}{h_\mathfrak{m}(E)}$, \toprec{p^M}, using Theorem~\ref{th:componeval} \textit{mutatis mutandis}. The cost of this computation is $\tO(h_\mathfrak{m}(E) (p^d d^{d+3} M^{d+2} c + \delta M \log p))$ bit operations. The result follows.
\end{proof}

\bibliographystyle{plain}
\bibliography{refs}

\end{document}